\documentclass{article}
\usepackage[utf8]{inputenc}
\usepackage{hyperref}
\usepackage[T1]{fontenc}
\usepackage[utf8]{inputenc}

\usepackage{amsthm}
\usepackage{amsmath}
\usepackage{amssymb}
\usepackage{amsfonts}
\usepackage{array}
\usepackage{graphicx}
\usepackage{bbold}
\usepackage{tikz}
\usepackage{circuitikz}

\usepackage{float}
\usepackage{enumitem} 
\usepackage[top=2cm, bottom=2cm, left=3cm, right=3cm]{geometry}
\usepackage{esvect}
\usepackage{latexsym}
 
\newenvironment{sproof}{%
  \proof}{\endproof}
\newtheorem{theorem}{Theorem}
\newtheorem*{theorem*}{Theorem}

\newtheorem{lemma}{Lemma}
\newtheorem{proposition}{Proposition}
\newtheorem*{proposition*}{Proposition}
\newtheorem{definition}{Definition}
\newtheorem{corollaire}{Corollary}
\newtheorem{property}{Property}
\newtheorem*{corollaire*}{Corollary}
\usepackage{algorithm}
\usepackage{algorithmic}
\newcommand{\footremember}[2]{%
    \footnote{#2}
    \newcounter{#1}
    \setcounter{#1}{\value{footnote}}%
}
\newcommand{\footrecall}[1]{%
    \footnotemark[\value{#1}]%
} 
\newcommand{\off}[1]{}

\newcommand{\R}{\mathbb{R}}
\newcommand{\N}{\mathbb{N}}
\newcommand{\bigO}{\mathcal{O}}
\def\argmin{\textup{argmin}\,}

\title{Study of the behaviour of Nesterov Accelerated Gradient in a non convex setting: the strongly quasar convex case}
\author{J. Hermant\footremember{Bdx}{Univ. Bordeaux, Bordeaux INP, CNRS, IMB, UMR 5251, F-33400 Talence, France}\footremember{corr}{Corresponding author: julien.hermant@math.u-bordeaux.fr} \and J.-F. Aujol\footrecall{Bdx} \and C. Dossal\footremember{Tls1}{IMT, Univ. Toulouse, INSA Toulouse, Toulouse, France} \and A. Rondepierre\footrecall{Tls1}}
\date{}

\begin{document}

\maketitle
\begin{abstract}
    We study the convergence of Nesterov Accelerated Gradient (NAG) minimization algorithm applied to a class of non convex functions called strongly quasar convex functions.
    We show that NAG can achieve an accelerated convergence speed at the cost of a lower curvature assumption. We provide a continuous analysis through high resolution ODEs, where we show that despite that negative friction may appear, the solution of the system achieves accelerated rate of convergence to the minimum. Finally, we identify the key geometrical property that, if dropped, theoretically cancels the acceleration phenomenon. 
\end{abstract}
\paragraph{Keywords:} Non-convex optimization, first order algorithms, strongly quasar convex, convergence rates, geometrical properties.
\section{Introduction}
We are interested in the following unconstrained minimization problem:
\begin{equation}\label{P}\tag{P}
    \min_{x \in \mathbb{R}^d} F(x):= F^\ast
\end{equation}
where $F: \mathbb{R}^d \to \mathbb{R} \cup \{ +\infty \}$ is such that $\arg \min_{x \in \mathbb{R}^d} F(x)$ is non empty. When it comes to minimize high dimensional functions, \textit{first order algorithms} such as gradient descent are popular, because of the relative cheapness of the iterations. These algorithms only make use of the function and its gradient, which are more computationally tractable than the Hessian that may be used by second order algorithms. We will study a specific type of first order algorithms called Nesterov Accelerated Gradient algorithms (NAG), which are variants of the gradient descent including an inertia mechanism. There exist classes of functions such that NAG algorithms allow to converge to the minimum $F^\ast$ with an accelerated speed of convergence compared to gradient descent. Considering convex functions, the convergence bound of the seminal version of (NAG) is $\bigO\left( \frac{1}{n^2} \right)$ \cite{Nesterov1983AMF}, improving over the $\bigO\left( \frac{1}{n} \right)$ rate of the gradient descent. When $F$ is $\mu$-strongly convex and $L$-smooth (\textit{i.e.} $C^1$ with a $L$-Lipschitz gradient), another version of NAG  \cite{nesterovbook} leads to an analogous acceleration phenomenon as we upgrade a $(1-\frac{\mu}{L})$ linear convergence rate into $(1-\sqrt{\frac{\mu}{L}})$, where $\frac{\mu}{L} \leqslant 1$ may be extremely low for high dimension functions. In recent applications, the problem of minimizing \textit{non convex} functions has become crucial, \textit{e.g.} in the field of machine learning. However, it is also known that the lack of regularity may cancel the acceleration of NAG. For example it has been shown that gradient descent is optimal among first order algorithms for the class of functions with a Lipschitz gradient \cite{lowerboundI}, see also \cite{PLlowerbound, RsigGdOptimal} for similar results on other classes of functions. This means that for some classes of functions, NAG converges theoretically not faster than GD. Convexity is however non necessary to get acceleration over gradient descent: for example \cite{convexguilty, momentumsaddle} show that a modified version of NAG accelerates over gradient descent for the class of functions with Lipschitz gradient and Lipschitz Hessian. Among convexity relaxations, the \textit{quasar convex} (originally weak-quasi convex \cite{quasarconvexorigin}) functions have gained a rising interest \cite{hinder2023nearoptimal,wang2023continuized,fu2023accelerated, quasarconvexinexxact, gower2021sgd}. These functions are defined by the following inequality:
\begin{equation}\label{quas conv intro}
           F^\ast \geqslant F(x) +  \frac{1}{\gamma} \langle\nabla F(x),x^\ast-x \rangle + \frac{\mu}{2} \lVert x^\ast - x \rVert^2
    \end{equation} with $x^\ast$ a minimizer, $x$ an arbitrary point belonging to $\mathbb{R}^d$, $\gamma \in (0,1],~ \mu \geqslant 0$. The case $\mu > 0$ defines \textit{strongly quasar convex} functions, which are the main focus of our work. Strongly quasar convex can be highly oscillating, while keeping some properties of strongly convex functions that are favorable from an optimization point of view. This induces that theses functions are interesting in order to study NAG's behaviour in a non convex setting. Moreover, it has been empirically observed  that the loss function of some neural networks has a quasar-convex like structure \cite{sgdQuasConvNeur}. 
\paragraph{Related work}
Convergence of momentum first order algorithms has attracted some interest. In the $L$-smooth and ($1,\mu)-$strongly quasar convex (\ref{quas conv intro}) setting, \cite{rungekutta} uses a Runge Kutta discretization procedure of the Heavy Ball ordinary differential equation to get a $ 1-(\frac{\mu}{L})^\gamma$ convergence rate, $ \gamma = \frac{s+3}{2(s+1)}$ where $s$ is such that the $s$th derivative is Lipschitz. This means that in this case the function is needed to be high order smooth to be close to the $1-\sqrt{\frac{\mu}{L}}$ accelerated rate. The authors of \cite{wang2023continuized} apply the \textit{continuized acceleration framework} \cite{even2021continuized} to $(\gamma,\mu)$-strongly quasar convex functions, which consists in a continuous stochastic differential equation approach leading to a stochastic version of the Nesterov accelerated Gradient. This stochastic version achieves an accelerated rate $1-\gamma\sqrt{\frac{\mu}{L}}$ in expectation, and convergence of iterations with high probability can be deduced. Also, \cite{quasarconvexold1,quasarconvexold2} show accelerated rates in term of number of iterations, but each iteration relies on a low dimensional sub-optimization problem to solve. In \cite{hinder2023nearoptimal}, the cost of this sub-problem (binary line search) is explicitly computed, and the authors show that their algorithm achieves an almost optimal rate (up to a log factor) for $L$ smooth $(\gamma,\mu)$-strongly quasar convex function in term of gradient and function evaluations. However, it is not clear whether NAG needs to solve a low dimensional sub-optimization problem at each iteration in order to achieve accelerated convergence.  Finally, in an arXiv preprint published after the first version of our work, \cite{gupta2024nesterov} address some similar questions. Using similar ideas as in \cite{hyppo}, they show convergence of smooth strongly quasar convex functions without assuming that the functions have a unique minimizer. However, their results hold in the restricted case $\gamma = 1$.

Among recent interpretations of NAG (\textit{e.g.} \cite{bubeck2015geometric, karimi2020linear}), an important one is the ODE framework, in which these algorithms are seen as a discretization of an ordinary differential equation \cite{POLYAK19641,suboydcandes}. One can show similar convergence results for the algorithms and for the solutions of these equations, mainly via Lyapunov approaches. These continuous analogs give interesting insights: it enables physical interpretation, convergence in this setting may be proved with less technical considerations than for the discrete counterpart, and importantly the strategies of proof may be adapted when we want to transpose results in the discrete setting, aiming to show algorithms convergence. This has been extensively used in a convex setting, see $\textit{e.g.}$ \cite{suboydcandes,siegel2021accelerated, aujdossrondPL,hyppo, attouch2020firstorder,shi2018understanding,nesterov1983linear}. However to the best of our knowledge, it has not been used yet in a quasar convex setting. In this paper, we will make a step in this direction using this framework to analyze convergence of strongly quasar convex functions.
\paragraph{Contributions}

We provide the following contributions.
\begin{enumerate}
    \item We identify a curvature assumption such that NAG provides an accelerated rate of convergence for the class of smooth strongly quasar convex functions. We extend this result to composite non differentiable functions, in which case specific difficulties appear. 
    % We argue that without the aforementioned assumption, one can hardly do without adaptive process to show global convergence, although NAG with well chosen parameters is expected to behave well in practice.
    \item We provide a continuous analysis of the high resolution ODE associated to NAG in the strongly quasar convex setting. We highlight limitations of the ODE framework as a tool to tune optimization algorithms when it comes to non convex optimization.
    \item By creating a link between strongly quasar convex functions and Polyak-{\L}ojasiewicz functions, we identify a key geometrical property that allows to achieve acceleration with NAG. Also, we present new properties of strongly quasar convex functions, together with properties considering more general non convex functions.
\end{enumerate}
\paragraph{Organisation of paper}
In section 2 we define the class of functions we study in this paper, and give a brief overview of their potential non convex behaviour. In section 3 we introduce the convergence of NAG applied to strongly quasar convex functions. In section 4 we present a continuous analysis of NAG in a strongly quasar convex setting. In Section 5 we present new properties of strongly quasar convex functions. In section 6 we illustrate our work with some numerical experiments. 
\section{Preliminaries}
Throughout the paper, we will mainly consider differentiable functions $F: \mathbb{R}^d \to \mathbb{R}$. 
A function $F: \mathbb{R}^d \to \mathbb{R}$ is said $L$-smooth for some $L>0$ if $F$ is $C^1$ and has a $L$-Lipschitz gradient: 
$$\forall  (x,y) \in \mathbb{R}^d\times \R^d,~\lVert \nabla F(x) - \nabla F(y) \rVert \leqslant L \lVert x-y \rVert.$$
Note that $F$ is $L$-smooth if and only if $F$ admits lower and upper quadratic bounds parameterized by $L$ at every point. More precisely:
\begin{property}
Let $F:\mathbb{R}^d\rightarrow \mathbb{R}$. $F$ is $L$-smooth for some $L>0$ if and only if it verifies for all $x,y$ in $\mathbb{R}^d$:
\begin{equation}\label{def L-smooth}
F(x) + \langle \nabla F(x),y-x \rangle - \frac{L}{2}\lVert x-y \rVert^2 \leqslant F(y) \leqslant  F(x) + \langle \nabla F(x),y-x \rangle + \frac{L}{2}\lVert x-y \rVert^2.
\end{equation}
\end{property}

%A differentiable function $F$ has a $L$-Lipschitz gradient if: 
%$$\forall  (x,y) \in \mathbb{R}^d\times \R^d,~\lVert \nabla F(x) - \nabla F(y) \rVert \leqslant L \lVert x-y \rVert.$$
%Having a $L$-Lipschitz gradient is equivalent to satisfy a $L$-smooth condition.
%\begin{property}
%Let $F:\mathbb{R}^d\rightarrow \mathbb{R}$. $F$ has a $L$-Lipschitz gradient if and only if it is a $L$-smooth function, %\textit{i.e.} it verifies for all $x,y$ in $\mathbb{R}^d$:
%\begin{equation}\label{def L-smooth}
%F(x) + \langle \nabla F(x),y-x \rangle - \frac{L}{2}\lVert x-y \rVert^2 \leqslant F(y) \leqslant  F(x) + \langle \nabla F(x),y-x \rangle + \frac{L}{2}\lVert x-y \rVert^2.
%\end{equation}
%\end{property}

\begin{proof}
    The fact that $L$-smooth implies the inequality \eqref{def L-smooth} is just the well-known descent lemma, see for example \cite{nesterovbook}. The converse is also well known when dealing with convex functions. Less trivially, it still holds in the non convex case. To the best of our knowledge the equivalence is not proved in the literature, but a proof has been proposed online\footnote{Characterization of {L}ipschitz derivative, \url{https://math.stackexchange.com/q/4264948} (version: 2021-10-01), Mathematics Stack Exchange} and is adapted here to our context.
    %he following proof is borrowed from \cite{DAW}, which also extend this equivalence result to the case of $L$-smooth and $L$-Lipschitz gradient properties are verified on an open convex subsets. To the best of our knowledge, this result does not exist in the literature, which motivates us to write the proof here.\\
    \\
    Assume first that $F$ verify (\ref{def L-smooth}) with $L=1$.
      Let $d\in \mathbb{R}^d$. Evaluating (\ref{def L-smooth}) at 4 different pairs of points:
    \begin{equation}
          F(y+d) - F(x) - \langle \nabla F(x),y+d-x\rangle \leqslant \frac{1}{2}\lVert y+d-x \rVert^2
    \end{equation}
    \begin{equation}
        F(x-d) - F(y) - \langle \nabla F(y),x-d-y\rangle \leqslant \frac{1}{2}\lVert y+d-x \rVert^2
    \end{equation}
    \begin{equation}
        -(F(y+d)-F(y) - \langle \nabla F(y),d \rangle)\leqslant \frac{1}{2}\lVert d \rVert^2
    \end{equation}
    \begin{equation}
             -(F(x-d)-F(x) - \langle \nabla F(x),-d \rangle)\leqslant \frac{1}{2}\lVert d \rVert^2.
    \end{equation}
    Adding all this inequalities yields:
\begin{equation}
    \langle \nabla F(x) - \nabla F(y),x-y-2d \rangle \leqslant \lVert y + d - x \rVert^2 + \lVert d \rVert ^2
\end{equation}
Set $g:= \nabla F(x) - \nabla F(y)$ and choose $d$ such that $x-y-2d = g$. Then $ d = \frac{1}{2}(x-y-g)$ and $y+d-x = -\frac{1}{2}(x-y+g)$. This results in:
\begin{equation}
    \lVert g \rVert^2 \leqslant \frac{1}{4}\lVert x-y+g \rVert^2 + \frac{1}{4} \lVert x-y-g \rVert^2 = \frac{1}{2}\lVert x-y \rVert^2 + \frac{1}{2}\lVert g \rVert^2
\end{equation}
which implies $F$ is with $1$-Lipschitz gradient.
It is straightforward to extend it to functions verifying (\ref{def L-smooth}) with arbitrary $L_0 \geqslant 0$, as it implies $\frac{F}{L_0}$ verify  (\ref{def L-smooth}) with $L=1$, inducing $\frac{\nabla F}{L_0}$ is $1$-Lipschitz, thus inducing the result.
\end{proof}
This provides quadratic lower and upper bounds on the function, both parameterized by the constant $L$. However, we will later consider different parameterizations for these bounds. To do so we now introduce the class of \textit{$(a,L)$-curvatured functions}. 
\begin{definition}\label{def:a_l_curvature}
Let $F:\R^d\rightarrow\R$ be a differentiable function and $(a,L)$ two real constants with $L>0$ and $a\leqslant L$. The function $F$ is said to be $(a,L)$-curvatured if it satisfies for all $x,y \in \mathbb{R}^d$:
\begin{equation}\label{curvature functions}
    F(x) + \langle \nabla F(x),y-x \rangle + \frac{a}{2}\lVert x-y \rVert^2 \leqslant F(y) \leqslant  F(x) + \langle \nabla F(x),y-x \rangle + \frac{L}{2}\lVert x-y \rVert^2.
\end{equation}
\end{definition}
This is a generalisation of $L$-smoothness as we allow for different characterizations of lower curvature. In particular, observe that a $(-L,L)$-curvatured function is exactly a $L$-smooth function, a $(0,L)$-curvatured function is a $L$-smooth convex function,  and a ($\mu,L$)-curvatured function with $\mu>0$ is a $L$-smooth and $\mu$-strongly convex function.\\
In this paper we will also consider the subclass of $C^2$ functions having a $\rho$-Lipschitz Hessian (for some $\rho \geqslant 0$), \textit{i.e.} functions that verify the following:
\begin{equation}
\forall (x,y)\in\mathbb{R}^d\times \R^d,~\left\vert \left\vert \left\vert \nabla^2 F(x) - \nabla^2 F(y)\right\vert \right\vert \right\vert\leqslant \rho \lVert x-y \rVert.
\end{equation}
where for the matrix norm $\vert \vert \vert . \vert \vert \vert$, we choose the norm induced by the euclidean norm on $\mathbb{R}^d$.
\subsection{Strong convexity and the question of acceleration}\label{section strong conv et question of acceleration}
In the paper, we study a relaxation of a well known property called strong convexity, which we recall below. Note that as we mainly studied the differentiable function case, we state the definitions in this case. In Section \ref{section non diff} we will consider possibly non differentiable functions.    
\begin{definition}
 Let $F:\R^d\rightarrow \R$ be a differentiable function. The function $F$ is said $\mu$-strongly convex for some $\mu > 0$ if it satisfies:
\begin{equation}\label{def strongly conv}
\forall (x,y)\in \mathbb{R}^d\times \R^d,~    F(x) + \langle \nabla F(x),y-x \rangle + \frac{\mu}{2}\lVert x-y \rVert^2 \leqslant F(y).
\end{equation}
\end{definition}
This is a stronger hypothesis than convexity (which corresponds to the above inequality with $\mu = 0$). Strongly convex functions are not only lower bounded by linear approximations, but by quadratic approximations. In particular, these functions do verify the $\mu$-quadratic growth hypothesis.
\begin{definition}\label{definition PL}
Let $F:\R^d\rightarrow \R$ such that $X^*=\argmin F\neq \emptyset$ and $F^*=\min F$. The function $F$ has $\mu$-quadratic growth for some $\mu>0$ if:
\begin{equation}
%  \forall x \in \mathbb{R}^d,~  \frac{\mu}{2}\lVert x-x^\ast \rVert^2 \leqslant F(x) - F^\ast.
\forall x \in \mathbb{R}^d,~  \frac{\mu}{2}d(x,X^\ast)^2 \leqslant F(x) - F^\ast.
\end{equation}
\end{definition}
It is straightforward to see that under the strong convexity assumption, the function $F$ has a unique minimizer. The quadratic growth property ensures that, around this minimizer, the function can not become flatter than a quadratic. This gives a control about how fast the gradient can vanish to zero when approaching the minimizer, allowing to achieve linear convergence speed with gradient descent. More precisely, for $L$-smooth and $\mu$-strongly convex functions, the gradient descent generates a sequence $\{ x_n\}_{n\in \mathbb{N}}$ that yields a linear convergence $F(x_n) - F^\ast \leqslant \bigO\left((1-\frac{\mu}{L})^n\right)$. As mentioned in introduction, using other first order algorithms, this rate can be improved into a $\bigO\left((1-\sqrt{\frac{\mu}{L}})^n\right)$ rate. Notice that we necessarily have $\mu \leqslant L$. Moreover, in large dimension, the ratio $\frac{\mu}{L}$ (which represents the inverse of the conditioning of the function to minimize) can be very small, this is actually a significant improvement.

\paragraph{Acceleration for relaxations of strong convexity}
In practice however few problems really suit the strong convexity hypothesis. This motivates to consider weaker assumptions. $\mu$-strong convexity provides a lower bound on the function, while $L$-smoothness provides an upper bound. Many relaxations of $\mu$-strongly convex and $L$-smooth functions generalize this fact by defining another pair of assumptions, a lower one parameterized by a constant $\mu \geqslant 0$ and an upper one parameterized by a constant $L\geqslant 0$. In many cases, the property $\mu \leqslant L$ and the convergence rate associated to gradient descent in $\bigO\left((1-\frac{\mu}{L})^n\right)$ remains true. This enables to generalize, in these cases, the characterization of \textit{acceleration} as the exchange of $\frac{\mu}{L}$ for $\sqrt{\frac{\mu}{L}}$. See \cite{guilleescuret2020study} for an insightful discussion about lower and upper conditions. 

An example of such relaxation is the class of aforementioned $L$-smooth functions having a $\mu$-quadratic growth. However this is a too weak relaxation, since we lose almost all control over the function. In particular critical points can be non (global) minimum.

This is why we consider slightly stronger hypotheses, such as the class of $\mu$-Polyak-\L ojasiewicz functions ($\mu$-PL), which is a non-convex relaxation of the class of $\mu$-strongly functions:
%\paragraph{Polyak-\L{}ojasiewicz functions}
%A stronger hypothesis, which is still a non convex relaxation of strong convexity, are $\mu$-Polyak-\L{}ojasiewicz ($\mu$-PL) functions.
\begin{definition}\label{PL}
Let $F:\R^d\rightarrow \R$ be a differentiable function with $X^*=\arg\min~F\neq\emptyset$ and $F^*=\min~F$. The function $F$ is $\mu$-Polyak-\L{}ojasiewicz ($\mu$-PL) for some $\mu>0$, if:
\begin{align}
 \forall x\in \mathbb{R}^d,~   F(x)-F^\ast \leqslant \frac{1}{2\mu}\lVert \nabla F(x) \rVert^2.
\end{align}
\end{definition}
Note that $\mu$-PL functions have a $\mu$-quadratic growth \cite{garrigos2023square}. The \L ojasiewicz property \cite{Loja63,Loja93} is a key tool in the mathematical analysis of continuous and discrete dynamical systems, initially introduced to prove the convergence of the trajectories for the gradient flow of analytic functions. The Polyak-\L ojasiewicz property is nothing more than the global version of the \L ojasiewicz property with an exponent $\frac{1}{2}$, and appears in important practical problems \cite{LossLandscape, peyreResnet}.

It has been shown in \cite{karimi2020linear} that gradient descent ensures, for $L$-smooth functions satisfying $\mu$-PL property, a linear convergence: $F(x_n) - F^\ast \leqslant (1-\frac{\mu}{L})^n(F(x_0)-F^\ast)$. Importantly, in \cite{PLlowerbound} is computed a lower bound of the number of gradient queries needed to achieve, with a first order method, a point $\hat{x}$ such that $F(\hat{x}) - F^\ast \leqslant \varepsilon (F(x_0) - F^\ast)$ for some $\varepsilon > 0$. They show that for every first order method, there exists a function such that this number of gradient queries is of the order $\frac{L}{\mu}\log\left(\frac{1}{\varepsilon}\right)$. This bound is achieved, up to a constant, by gradient descent. Strikingly, it induces that for these functions, the Nesterov accelerated gradient algorithms are prevented to achieve an accelerated linear convergence of the form $F(x_n) - F^\ast \leqslant K_1(1-K_2 \sqrt{\frac{\mu}{L}})^n(F(x_0)-F^\ast)$, where $K_1>0$ and $0 < K_2\le1$ constants independent of $\mu$ and $L$. This motivates to use stronger assumptions in order to achieve this acceleration, and leads us to consider a stronger hypothesis, namely the strong quasar convexity, which is another relaxation of strong convexity.

\subsection{A relaxation of strong convexity: strong quasar convexity}\label{quasar convex sction 2}
In this section we define the notion of strong and non-strong quasar convexity. We include a short discussion about its geometrical properties, which will be further developed in section \ref{section geometrical considerations}.
\begin{definition}\label{definition (strongly) quasar convex}
    Let $F:\R^d\rightarrow \R$ a differentiable function such that $X^*=\argmin F\neq \emptyset$ and $F^*=\min~F$. Let $x^*$ be a minimizer of $F$ and $\gamma \in (0,1]$, $\mu > 0$. The function $F$ is said $\gamma$-quasar convex with respect to $x^\ast$ if it satisfies
     \begin{equation}\label{QC}
      \forall x\in \R^d,~   F^\ast \geqslant F(x) +  \frac{1}{\gamma} \langle\nabla F(x),x^\ast-x \rangle,
    \end{equation}
    and $(\gamma,\mu)$-strongly quasar convex with respect to $x^\ast$ if:
    \begin{equation}\label{ineq strongly quasar convex}
           \forall x\in \R^d,~      F^\ast \geqslant F(x) +  \frac{1}{\gamma} \langle\nabla F(x),x^\ast-x \rangle + \frac{\mu}{2} \lVert x^\ast - x \rVert^2.
    \end{equation}
\end{definition}
We refer to any minimizer $x^*$ at which \eqref{QC} holds, as a quasar-convex point of $F$. The class of (strongly) quasar convex functions was first introduced in 2017 by \cite{quasarconvexorigin} (with $\gamma > 0$) where the authors refer to it as {\it weak quasi-convexity}. It was then implicitly re-used by \cite{suboydcandes} and \cite{aujol2019optimal,apidopoulos2021convergence} as a {\it flatness condition} (with $\gamma \geqslant 1$), and revisited more recently in \cite{hinder2023nearoptimal} where the quasar convexity name was introduced. A nice property of this class of functions is that any critical point of (strongly) quasar convex function $F$ is a global minimizer of $F$.

Moreover, the set of minimizers $X^*$ of a quasar convex function has some strong regularity: it is a star convex set i.e. there exists $x^*\in X^*$ such that:
\begin{equation}
\forall x\in X^*,~\forall t\in[0,1],~tx^*+(1-t)x \in X^*,\label{def:starconvex}
\end{equation}
and is reduced to a single for strongly quasar convex functions, see \cite[Appendix D, Observations 3 and 4]{hinder2023nearoptimal}.

Lastly, observe that the Polyak-\L ojasiewicz and the quadratic growth properties can be seen as relaxations of strong quasar convexity, as stated by the following result: 
\begin{proposition}\label{SQC implis PL & QG}
   Let $F:\R^d\rightarrow \R$ be a $(\gamma,\mu)$-strongly quasar convex function for some $(\gamma,\mu) \in (0,1] \times \mathbb{R}_+$ and $x^*$ its minimizer. Let $F^*=\min~F$. Then:
    \begin{enumerate}

        \item $F$ is $\mu \gamma^2$-PL, \textit{i.e.}
        \begin{align}
          \forall x\in \R^d,~  \frac{1}{2\gamma^2 \mu}\lVert \nabla F(x)\rVert^2 \geqslant F(x) - F^\ast.
        \end{align}
        \item \cite[Corollary 1]{hinder2023nearoptimal} $F$ has a $\frac{\gamma \mu}{2-\gamma}$-quadratic growth, \textit{i.e}
        \begin{align}
            \forall x\in \R^d,~ F(x) - F^\ast \geqslant \frac{\gamma \mu}{2(2-\gamma)}\lVert x-x^\ast \rVert^2.
        \end{align}
    \end{enumerate}
\end{proposition}
\begin{figure}[ht]
    \centering
    \includegraphics[scale=0.5]{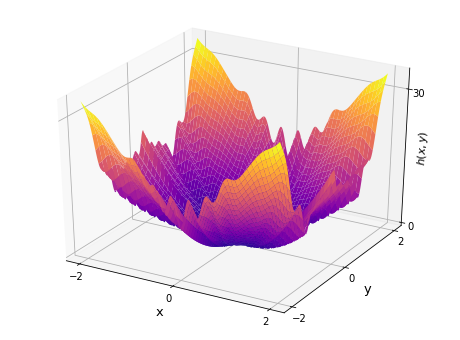}
    \includegraphics[scale=0.4]{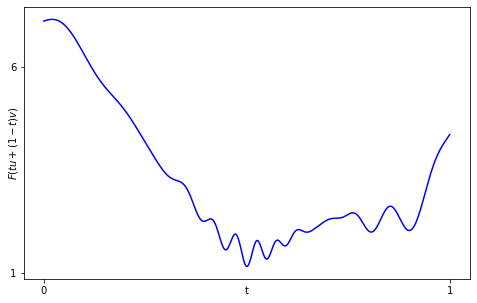}
    \caption{An example of strongly quasar convex function built as (\ref{synthetic sqc}), whose explicit expression is given in section \ref{appendix numerical}. On the left, the graph of this function. On the right, a cut of this graph along a segment of $\mathbb{R}^2$, such that the minimizer does not belong to this segment.}
    \label{fig:1}
\end{figure}
 Strongly quasar convex functions can be highly oscillating. Let us take the construction described in \cite{lee2016optimizingstarconvex,hinder2023nearoptimal} to highlight the non-convexity of strongly quasar convex functions. Let $f: \mathbb{R} \to \mathbb{R}$ such that it is $(\gamma,\mu)$-strongly quasar convex, with $f^\ast = f(0) = 0$. Let $g: S^{d-1} \to \mathbb{R}$ be an arbitrary continuous function defined on the unit circle of $\mathbb{R}^d$ such that $g \geqslant 1$. Consider 
\begin{equation}\label{synthetic sqc}
    h(x) = f(\lVert x \rVert)g\left(\frac{x}{\lVert x \rVert}\right), ~ x\in \mathbb{R}^d.
\end{equation}
This function is $(\gamma,\mu)$-strongly quasar convex independently of the choice of $g$ (see \cite[Appendix D.3]{hinder2023nearoptimal} for the non strongly quasar convex case, and see Appendix \ref{appendix synthetic} for the strongly quasar convex case). An example of such a function is displayed on the left side of Figure \ref{fig:1}. Radially this function behaves like $cf(\lVert x \rVert)$ where $c$ is constant. Restricted to this direction, the function is unimodal and critical points are minimizers \cite[Observation 1]{hinder2023nearoptimal}. However since $g$ may be extremely non convex, we see that taking the segment between two arbitrary points $x_0$ and $x_1$, smoothness aside we will have no control over the behaviour of the function (\textit{e.g.} right side of Figure \ref{fig:1}).\\
This lack of local regularity may not be a big deal with gradient descent, as it follows a descent direction at each iteration. However, in the case of NAG, the presence of inertia prevents from controlling the direction of the trajectory. We will see later that this lack of local regularity complicates considerably the potential acceleration of NAG.
\section{Acceleration with curvature assumption for strongly quasar convex functions}
\subsection{The gradient descent convergence}
As we are interested in faster algorithms than gradient descent, we first set the convergence results associated with this algorithm. We recall it is defined for some $x_0 \in \mathbb{R}^d$, $s\geqslant0$ by this recursive formula:
\begin{equation}\label{gradient descent}\tag{GD}
    x_{n+1} = x_n - s \nabla F(x_n).
\end{equation}
Let us first recall a known result for smooth strongly convex functions. 
\begin{proposition}[\cite{nesterovbook}]\label{prop gd sc}
    Let $F:\R^d\rightarrow \R$ be a $L$-smooth and $\mu$-strongly convex function for some $0<\mu \leqslant L$, and $x^*$ its minimizer. Let $F^*=\min~F$. Let $(x_n)_{n \in \mathbb{N}}$ be generated by (\ref{gradient descent}) with stepsize $s = \frac{1}{L}$. Then:
    \begin{equation}
    \forall n\in \N,~\lVert x_n-x^\ast \rVert^2 \leqslant \left(1-\frac{\mu}{L}\right)^{n}\lVert x_0-x^\ast \rVert^2.
    \end{equation}
\end{proposition}
Proposition~\ref{prop gd sc} indicates that for $\mu$-strongly convex and $L$-smooth functions, (\ref{gradient descent}) achieves a linear decrease.
Proposition~\ref{prop gd sc} can be extended to strongly quasar convex functions. To the best of our knowledge, the following proposition is not clearly stated in literature. The case $\gamma = 1$ is proved in \cite{guilleescuret2020study}. A result can be found in the stochastic case in \cite{gower2021sgd}, from which a result for the deterministic case can be deduced. Their stepsize is however lower than ours, so it yields a slower rate of convergence.
\begin{proposition}\label{prop gd qsc}
    Let $F:\R^d \rightarrow \R$ be a $L$-smooth and $(\gamma,\mu)$-strongly quasar convex function for some $0<\mu \leqslant L$, $\gamma \in (0,1]$, and let $x^*$ be its minimizer. Let $(x_n)_{n \in \mathbb{N}}$ be generated by (\ref{gradient descent}) with stepsize $s \leqslant \frac{1}{L}$. Then:
    \begin{equation}
    \forall n\in \N,~F(x_n) - F^\ast \leqslant \frac{2}{\gamma}(1 - \gamma\mu s)^n(F(x_0) - F^\ast).
    \end{equation}
\end{proposition}
%\begin{sproof}
%    The proof aims to show that the Lyapunov function $E_n = F(x_n) - F^\ast + \frac{\mu}{2}\lVert x_n - x^\ast \rVert^2$ satisfies $E_{n+1} \leqslant (1-\gamma \mu s)E_n$. It is deferred in appendix \ref{appendix GD}.
%\end{sproof}
%
\begin{proof}
Let $x^*$ be the quasar convex point of $F$ and:
\begin{equation}
        E_n = F(x_n) - F^\ast + \frac{\mu}{2}\lVert x_n - x^\ast \rVert^2.
    \end{equation}
We compute
\begin{align}
    E_{n+1} - E_n &= F(x_{n+1})-F(x_n) + \frac{\mu}{2}\lVert x_{n+1} - x^\ast \rVert^2 - \frac{\mu}{2}\lVert x_n - x^\ast \rVert^2\\
    &\overset{(\ref{gradient descent})}{=}F(x_{n+1})-F(x_n) -\mu s \langle x_n - x^\ast,\nabla F(x_n) \rangle + s^2 \frac{\mu}{2} \lVert \nabla F(x_n) \rVert^2.
\end{align}
The $L$-smooth inequality (\ref{def L-smooth}) implies that: $F(x_{n+1}) - F(x_n) \leqslant -\frac{s}{2}\lVert \nabla F(x_n) \rVert^2$ provided that $s\leqslant \frac{1}{L}$. Combined with the $(\gamma,\mu)$-strongly quasar convexity to control the scalar product, we get:
\begin{align}
    E_{n+1} - E_n &\leq-\frac{s}{2}\lVert \nabla F(x_n) \rVert^2- \gamma\mu s (F(x_n)-F^\ast) -\gamma s \frac{\mu^2}{2}\lVert x_n - x^\ast \rVert^2 + s^2 \frac{\mu}{2} \lVert \nabla F(x_n) \rVert^2\\
    &= \frac{s}{2}\left( \mu s-1 \right)\lVert \nabla F(x_n) \rVert^2 - \gamma \mu s \left( F(x_n) - F^\ast + \frac{\mu}{2}\lVert x_n - x^\ast \rVert^2 \right)
\end{align}
as $s \leqslant \frac{1}{L}$ the first term is negative, inducing 
\begin{equation}
    E_{n+1} - E_n \leqslant -\gamma \mu s E_n \Rightarrow E_{n+1} \leqslant (1-\gamma \mu s)E_n.
\end{equation}
By induction, we then deduce: 
 \begin{equation}
    \forall n\in \N,~     E_n \leqslant (1 - \gamma\mu s)^nE_0.
    \end{equation}
    By definition of $E_0$, and because $E_n \geqslant F(x_n) - F^\ast$, it follows that
    \begin{equation}
        F(x_n) - F^\ast \leqslant  (1 - \gamma\mu s)^n(F(x_0) - F^\ast + \frac{\mu}{2}\lVert x_0 - x^\ast \rVert^2).
    \end{equation}
Using the $\frac{\gamma \mu}{2-\gamma}$-quadratic growth induced by the $(\gamma,\mu)$-quasar strong convexity of $F$ (see Corollary 1 \cite{hinder2023nearoptimal}), we finally get the expected convergence rate.
\end{proof}

\subsection{The Nesterov Accelerated Gradient convergence}\label{sec:sc}
The Nesterov accelerated gradient algorithm used to optimize $L$-smooth and $\mu$-strongly convex function (NAG-SC) is often written in the following way \cite[Algorithm (2.2.22)]{NesterovBook2018}:
\begin{equation}\label{NAG-SC 2 POINTS}\tag{NAG-SC 2 POINTS}
    \left\{
    \begin{array}{ll}
        y_n = x_n + \frac{1 - \sqrt{\frac{\mu}{L}}}{1+\sqrt{\frac{\mu}{L}}} (x_n - x_{n-1}) \\
        x_{n+1} = y_n - \frac{1}{L}\nabla F(y_n)
    \end{array}
\right.
\end{equation}
Introducing the auxiliary variable:
$$z_n = \left(1+\sqrt\frac{L}{\mu}\right)y_n -\sqrt\frac{L}{\mu}x_n,$$
this algorithm can be rewritten as a $3$-points scheme, see \cite[Algorithm (2.2.19) with $\gamma_0 = \mu$]{NesterovBook2018}:
\begin{equation}\label{NAG-SC 3 POINTS}\tag{NAG-SC 3 POINTS}
    \left\{
    \begin{array}{ll}
        y_n = \left( \frac{1}{1 + \sqrt{\frac{\mu}{L}}} \right) x_n + (1- \frac{1}{1 + \sqrt{\frac{\mu}{L}}})z_n \\
        x_{n+1} = y_n - \frac{1}{L}\nabla F(y_n)\\
        z_{n+1} = (1 - \sqrt{\frac{\mu}{L}})z_n + \sqrt{\frac{\mu}{L}}(y_n - \frac{1}{\mu} \nabla F(y_n))
    \end{array}
\right.
\end{equation}
We have the following well known result.
\begin{theorem}[\cite{nesterovbook}]\label{th:cv_sc}
    Let $F:\R^d\rightarrow \R$ be $L$-smooth and $\mu$-strongly convex function for some $0<\mu \leqslant L$, and $x^*$ its unique minimizer. Let $F^* = \min~F$. Let $(x_n)_{n \in \mathbb{N}}$ be the sequence generated by (\ref{NAG-SC 3 POINTS}) with $x_0 = z_0$. Then:
    \begin{equation}
      \forall n\in \N,~  F(x_n) - F^\ast \leqslant  \left(1 - \sqrt{\frac{\mu}{L}}\right)^n\left( F(x_0)-F^\ast + \frac{\mu}{2}\lVert x_0 - x^\ast \rVert^2\right).
    \end{equation}
\end{theorem}
Observe that this is a considerable improvement over the $\bigO\left( (1-\frac{\mu}{L})^n \right)$ rate given by gradient descent (stepsize $\frac{1}{L}$), as $\frac{\mu}{L} \leqslant 1$ may be extremely low for high dimensional functions. 
% This phenomenon of trading $\frac{\mu}{L}$ for $\sqrt{\frac{\mu}{L}}$ is what we call \textit{acceleration} for functions parameterized by ($\mu$,$L$) where $\mu$ correspond to a lower condition ($\textit{e.g.}$ strong convexity here) and $L$ an upper condition ($\textit{e.g.}$ $L$-smoothness here). This is because often, the properties $\mu \leqslant L$ and the rate $F(x_n)-F^\ast \leqslant O\left( (1-\frac{\mu}{L})^n \right)$ for gradient descent remains true. See \cite{guilleescuret2020study} for an interesting discussion about lower and upper conditions.

Let us introduce a generalization of the classical (\ref{NAG-SC 3 POINTS}) algorithm, see Algorithm~\ref{algo} (see \textit{e.g.} \cite{hinder2023nearoptimal}).
\begin{algorithm}[H]
\caption{Nesterov Accelerated Gradient (3 points form)}\label{algo}
\begin{algorithmic} 
\STATE Let $z_0 = x_0$
\FOR{n = 0,\dots,}
\STATE $y_n = \alpha_nx_n + (1- \alpha_n)z_n$
\STATE $x_{n+1} = y_n - s\nabla F(y_n)$
\STATE $z_{n+1} = \beta_nz_n + (1 - \beta_n)y_n - \eta_n \nabla F(y_n)$
\ENDFOR
\end{algorithmic}
\end{algorithm}
 % This algorithm can be seen as a simultaneous minimization of a lower and upper approximation \cite{allen2014linear}.
 We want to use Algorithm~\ref{algo} to achieve an accelerated convergence similar to Theorem~\ref{th:cv_sc} with strongly quasar convex functions. We do so using the notion of \textit{curvature function} (\ref{curvature functions}):
 
\begin{theorem}\label{theorem 1}
    Let $F:\R^d\rightarrow \R$ be a $(\gamma,\mu)$-strongly quasar convex function for some $(\gamma,\mu)\in (0,1]\times \R_+^\ast$ and $F^*=\min~F$. Assume additionally that $F$ is a $(\rho,L)$-curvatured function for some $L>0$ and $\rho\leqslant L$. Let $(x_n)_{n \in \mathbb{N}}$ be a sequence of iterates generated by Algorithm 1 with parameters:
    \begin{center}
        $s\leqslant \frac{1}{L}$, $\alpha_n = \frac{1}{1 + \sqrt{\mu s}}$, $\beta_n = 1 - \gamma \sqrt{\mu s}$, $\eta_n = \frac{\sqrt{s}}{\sqrt{\mu }}$.
    \end{center} 
    If $\rho \geqslant -\gamma \sqrt{\frac{\mu}{s}}$, then:
        \begin{equation}
       \forall n\in \N,~ F(x_{n}) - F^\ast \leqslant \frac{2}{\gamma} \left( 1 - \gamma \sqrt{\mu s} \right)^n(F(x_0)-F^\ast).
    \end{equation}
\end{theorem}
See the proof in Appendix \ref{appendix proof theorem 1}, where we prove linear decrease of the following Lyapunov function
\begin{equation}
    E_n = F(x_n) - F^\ast + \frac{\mu}{2}\lVert z_n - x^\ast \rVert^2.
\end{equation}
More precisely, we show $E_{n+1} \leqslant (1-\gamma \sqrt{\frac{\mu}{L}})E_{n}$, $\forall n \in \mathbb{N}$. The $(\rho,L)$-curvature assumption with $\rho<0$ allows us to replace convexity by a weaker bound control, namely:
$$\forall n\in \N, ~\langle \nabla F(y_n),x_n-y_n \rangle + F(y_n)  - F(x_n)  \leqslant  -\frac{\rho}{2}\lVert x_n - y_n \rVert^2.$$

Fixing $s = \frac{1}{L}$, the curvature bound becomes $\rho \geqslant -\gamma \sqrt{\mu L} = -\gamma \sqrt{\frac{\mu}{L}} L$. Parameterizing the algorithm with an arbitrary step size $s\leqslant \frac{1}{L}$ allows us to highlight that we can trade restriction on the curvature with convergence speed. Formally:
\begin{corollaire}
    Let $F:\R^d\rightarrow \R$ be a $L$-smooth and $(\gamma,\mu)$-strongly quasar convex function for some $0<\mu \leqslant L$, $\gamma \in (0,1]$, and let $F^*=\min~F$. Let $(x_n)_{n \in \mathbb{N}}$ generated by Algorithm 1 with parameters $s= \gamma^2 \frac{\mu}{L^2}$, $\alpha_n = \frac{1}{1 + \sqrt{\mu s}}$, $\beta_n = 1 - \gamma \sqrt{\mu s}$ and $\eta_n = \frac{\sqrt{s}}{\sqrt{\mu }}$. Then:
           \begin{equation}
        \forall n\in \N,~F(x_{n}) - F^\ast \leqslant \frac{2}{\gamma} \left( 1 - \gamma^2 \frac{\mu}{L} \right)^n(F(x_0)-F^\ast).
    \end{equation}\label{cor:SQC}
\end{corollaire}
\begin{proof}
    Just solve $-\gamma \sqrt{\frac{\mu}{s}} = -L$, we get $s = \gamma^2\frac{\mu}{L^2}$.
\end{proof}
We see that we can delete the curvature assumption, but at the cost of acceleration. The negative curvature problem has already been observed under other non convex hypothesis. Authors in \cite{convexguilty, momentumsaddle} prove acceleration convergence to a critical point in a Hessian Lipschitz setting, where they use NAG when curvature between iterates is not too negative, otherwise an alternative step using Hessian Lipschitz property is performed. In our case, alternative step using Hessian Lipschitz property is hardly useful as our Lyapunov designed to obtain linear convergence is less adapted to this argument.

\paragraph{Comparison with the algorithm of \cite[Algorithm 3]{hinder2023nearoptimal} (binary search)} The convergence bound of Theorem~\ref{theorem 1} is similar to the one of \cite{hinder2023nearoptimal}. They achieve this bound using Algorithm~\ref{algo} where, at each iteration, a line search procedure to compute one of the parameter (see Algorithm 3 in \cite{hinder2023nearoptimal}). This results in a $\log\left(\gamma^{-1}\frac{L}{\mu}\right)$ factor added to the number of gradient and function evaluations needed, that does not appear in our convergence rate. Getting rid of this line search procedure allows to get deeper insights about NAG using ODEs modeling, as we see in Section~\ref{sec:continu}.
\paragraph{2 points version form of algorithm 1} Following arguments of \cite{lee2021geometric}, we can rewrite our algorithm in a 2 points version, as stated with this result:
\begin{proposition}\label{2 points scheme}
The algorithm 1 with parameters $s\leqslant \frac{1}{L}$, $\alpha_n = \frac{1}{1 + \sqrt{\mu s}}$, $\beta_n = 1 - \gamma \sqrt{\mu s}$ and $\eta_n = \frac{\sqrt{s}}{\sqrt{\mu }}$ can be written as the following 2 points algorithm:
    \begin{equation}\label{NAG-QSC-2 points}\tag{NAG-SQC 2 POINTS}
    \left\{
    \begin{array}{ll}
        y_n = x_n + \frac{1 - \gamma \sqrt{\mu s}}{1 + \sqrt{\mu s}} (x_n - x_{n-1}) + \frac{\sqrt{\mu s}}{1 + \sqrt{\mu s}}(\gamma -1 )(x_n - y_{n-1}) \\
        x_{n+1} = y_n - s\nabla F(y_n)
    \end{array}
\right. 
\end{equation}
\end{proposition}

When $\gamma = 1$, we recover the classical Nesterov algorithm for minimizing strongly convex functions. Note that 2 points scheme versions are widely present in the literature. Considering the 2 points or the 3 points version gives different interpretations of the algorithm, and switching from a version to another is not obvious. %For instance, the famous FISTA \cite{beckTeboulle}, which extends Nesterov's algorithm to a class of non differentiable functions, is formulated as a 2 points scheme.
The proof of Proposition~\ref{2 points scheme} is in Appendix \ref{appendix 2 points}.

% \paragraph{Non strongly quasar convex case} It appears that in the quasar convex case, one can hardly use restriction on curvature arguments. An argument using ODE framework is given in Section \ref{section continuous discrete rupture}, and we provide a relatively short analysis of Algorithm 1 in the quasar convex case in Appendix \ref{appendix quasar convex nesterov}.
\subsection{Composite non differentiable case}\label{section non diff}
%{\color{purple} Développer un peu plus l'intro de cette partie ? A discuter: étoffer un peu cette partie en reprenant quelques éléments en annexe (D.2.1 et D.2.2.) de façon synthétique. En gros faire une première partie (numérotée ou pas) sur la quasar convexité par rapport à un non-minimiseur: reprendre les explications sur le fait que le QSC ne suffira pas, donner la def de la quasar-convexité par rapport à un point qcq, expliquer comment on peut construire des fonctions SQC par rapport à un non minimiseur en donnant une proposition regroupant quelques résultats à sélectionner dans les annexes. Puis dans un second temps énoncer le théorème 3.
%We give in this section a weaker statement of Theorem \ref{theorem 1} in a non differentiable setting, and we discuss the {\color{purple}impracticability} of quasar convex composite minimization.}

A natural way to extend the differential case to the non differentiable one is to consider the class of composite functions $F = f + g$, where $f$ is assumed to be differentiable and strongly quasar convex, and $g$ convex, proper lower semi-continuous, and such that its proximal operator \cite{combettes2010proximal} can be computed:
\begin{equation}
    \text{prox}_{g}(x) = \arg \min_y \left(g(y) + \frac{1}{2}\lVert x-y \rVert^2 \right).
\end{equation}
Typical example of such functions are the composite functions with a $\ell_1$ penalisation term: $F=f+\lVert x \rVert_1$.
%$(-\gamma \sqrt{\frac{\mu}{s}},L)$-curvatured, $C^1$ quasi strongly convex (\textit{i.e.} ($1$,$\mu$)-strongly quasar convex) and such that $g$ is convex, proper lower semi-continuous and non necessarily differentiable. A typical example of such a function could be $F = f + \lVert x \rVert_1$ where $f$ is as described above. We want to extend Theorem \ref{theorem 1} to this class of function. We will make use of the proximal operator \cite{combettes2010proximal}: \begin{equation}
%    \text{prox}_{g}(x) = \arg \min_y \left(g(y) + \frac{1}{2}\lVert x-y \rVert^2 \right)
%\end{equation}
%which we suppose we are able to compute.
The idea is then to replace the gradient in Algorithm 1 by the composite gradient mapping:
\begin{equation}
   \nabla f(y_n) \rightarrow \frac{1}{s}\left( y_n - \text{prox}_{sg}(y_n - s \nabla f(y_n)) \right)
\end{equation}
which is a generalization of the gradient (we recover it when $g \equiv 0$). It leads us to Algorithm 2.
     \begin{algorithm}
\caption{Proximal Nesterov Accelerated Gradient (3 points form)}
\begin{algorithmic} \label{algo non diff}
\STATE Soit $z_0 = x_0$
\FOR{k = 0,\dots,}
\STATE $y_n = \alpha_nx_n + (1- \alpha_n)z_n$
\STATE $x_{n+1} = \text{prox}_{sg}(y_n - s \nabla f(y_n)):= T_s(y_n))$
\STATE $z_{n+1} = \beta_nz_n + (1 - \beta_n)y_n - \frac{\eta_n}{s}(y_n - T_s(y_n)) $
\ENDFOR
\end{algorithmic}
\end{algorithm}

Extending Theorem \ref{theorem 1} to the non-differentiable case introduces an additional technicality since a minimizer/quasar point of $f$ is generally not a minimizer of the composite function $F$. In other words, if $f$ is strongly quasar convex with respect to its minimizer $x^\ast_f$, there is clearly no guarantee that this assumption holds at a minimizer $x^\ast_F$ of $F$. To enforce the latter property, we propose an extension of strong quasar convexity with respect to another point than a minimizer as suggested in \cite[Appendix D.2]{hinder2023nearoptimal}:
%If we set $f$ satisfying (\ref{definition (strongly) quasar convex}), an issue appears while trying to prove convergence: by Definition \ref{definition (strongly) quasar convex} $f$ is strongly quasar convex with respect to its minimizer $x^\ast_f$, but the hypothesis does not hold at $x^\ast_F$, the minimizer of $F$. We have no control on the relation between $x^\ast_f$ and $x^\ast_F$, and if we try to circumvent it by setting $x^\ast_f = x^\ast_F$ the problem loses all interest because in that case one may simply minimize $f$. \\
%It is yet possible to suppose that $f$ is strongly quasar convex at $x^\ast_F$, the minimizer of $F$. {\color{red}Note that the possibility to take another point than the minimizer for this assumption is mentioned in \cite{hinder2023nearoptimal}}.
\begin{definition}
Let $f:\R^d\rightarrow \R$ and $\hat x\in \R^d$. The function $f$ is said $(1,\mu)$-strongly quasar convex with respect to $\hat{x}$ if:
\begin{equation*}
\forall x\in \R^d,~\forall t\in [0,1],~f(t \hat{x} + (1-t)x) + t\left(1-t \right)\frac{\mu}{2}\lVert \hat x - x \rVert^2 \leqslant t f(\hat{x}) + (1-t)f(x),\label{def:QSC:nondiff}
\end{equation*}
or equivalently, when $f$ is additionally assumed to be differentiable:
    \begin{equation}
     \forall x\in \R^d, ~    f(\hat x) \geqslant f(x) + \langle\nabla f(x),\hat{x}-x \rangle + \frac{\mu}{2} \lVert \hat{x} - x \rVert^2.
    \end{equation}\label{def:SQC:nonminimizer}
\end{definition}
The equivalence is showed in \cite[Lemma 11]{hinder2023nearoptimal}, whose proof works for $\hat{x}$ not being a minimizer, if we consider $(\gamma,\mu)$ strong convexity with $\gamma = 1$. 
%{\blue Je ne comprends pas cette remarque - In this definition we restricted ourselves to $\gamma = 1$, because otherwise we have to restrict ourselves over a subset $\{ x, f(x) \geqslant f(\hat{x}) \}$}.\\
\paragraph{The $\gamma < 1$ case} 
Taking $\gamma < 1$ does not cope well with quasar convexity with respect to an arbitrary point.
Assume that for some $f \in C^1$, $\mu >0$ and $\gamma \in (0,1]$ the following holds:
    \begin{equation}\label{def:SQC:nonminimizergamma}
     \forall x\in \R^d, ~    f(\hat x) \geqslant f(x) +\frac{1}{\gamma} \langle\nabla f(x),\hat{x}-x \rangle + \frac{\mu}{2} \lVert \hat{x} - x \rVert^2.
    \end{equation} with $\hat{x}$ such that $\nabla f(\hat{x}) \neq 0$. 
Set $x_h = \hat{x} - h \nabla f(\hat{x})$, $h >0$. Using a order $1$ Taylor-Young development:
\begin{equation}
    f(\hat{x}) = f(x_h) + \langle \nabla f(x_h),\hat{x}-x_h \rangle + o(\lVert \hat{x}-x_h\rVert) = f(x_h) + h \langle \nabla f(x_h),\nabla f(\hat{x}) \rangle + o(h)
\end{equation}
Combining this with (\ref{def:SQC:nonminimizergamma}) evaluated in $x = x_h$, we have:
\begin{equation}
    \langle \nabla f(x_h), \nabla f(\hat{x}) \rangle \left( 1- \frac{1}{\gamma} \right) + \frac{o(h)}{h} \geqslant \frac{h\mu}{2}\lVert \nabla f(\hat{x}) \rVert^2
\end{equation}
As $h$ goes to $0$, the right hand side goes to $0$ while the left hand side goes to $\lVert \nabla f(\hat{x}) \rVert \left( 1- \frac{1}{\gamma} \right)$ because $f$ is $C^1$. As the latter is strictly negative for $\gamma \in (0,1)$, then in our definition $\gamma$ must be equal to $1$. 
\\
Observe now that Definition \ref{def:SQC:nonminimizer} is not an empty definition, and that such functions can easily be defined. A possible construction consists in adding some convex function to a $(1,\mu)$-strongly quasar convex function.
\begin{lemma}\label{lemma sqc not minimize}
  Let f: $\mathbb{R}^d \to \mathbb{R}$ be $(1,\mu)$-strongly quasar convex with respect to $\hat{x}$ and $g: \mathbb{R}^d \to \mathbb{R}$ be convex, both non necessarily differentiable. Then $F:= f + g$ is $(1,\mu)$-strongly quasar convex with respect to $\hat{x}$.
\end{lemma}
\begin{proof}
The proof is straightforward by summing the respective inequalities of $(1,\mu)$-strong quasar convexity of $f$ and convexity of $g$ written between any $x\in \R^d$ and $\hat x$.  
\end{proof}

%We prove this lemma in Appendix \ref{appendix non diff}, where we use a non differentiable characterisation of strong quasar convexity borrowed to \cite{hinder2023nearoptimal}.
As a special case of Lemma \ref{lemma sqc not minimize}, observe that if $f$ is $(1,\mu)$-strongly quasar convex with respect to its minimizer $x^\ast_f$, then $x^\ast_f$ is not necessarily a minimizer of $F$, but it is a quasar convex point of $F$.    

\paragraph{Convergence result for Algorithm \ref{algo non diff}} Using the extended definition of $(1,\mu)$-strongly quasar convexity, we are ready to state our result.
    \begin{theorem}\label{theorem 2}
    Let $F = f+g$ where $f: \mathbb{R}^d \to \mathbb{R}$ is a $L$-smooth function for some $L>0$ and $g: \mathbb{R}^d \to \mathbb{R}$ is convex, proper, lower semi-continuous. Assume that $F$ has a non empty set of minimizers and that $f$ is $(1,\mu)$-strongly quasar convex with respect to $x^\ast_F \in \arg \min~F$ for some $0<\mu \leqslant L$ and $(\rho,L)$-curvatured for some $\rho\leqslant L$.
Let $(x_n)_{n \in \mathbb{N}}$ be the sequence of iterates generated by Algorithm \ref{algo non diff} with parameters:
\begin{center}
    $s\leqslant \frac{1}{L}$, $\alpha_n = \frac{1}{1 + \sqrt{\mu s}}$, $\beta_n = 1 -  \sqrt{\mu s}$, $\eta_n = \frac{\sqrt{s}}{\sqrt{\mu }}$.
\end{center} If $\rho\geqslant -\sqrt{\frac{\mu}{s}}$, then:
\begin{equation}
     \forall n\in \N,~   F(x_{n}) - F^\ast \leqslant 2 \left( 1 - \sqrt{\mu s} \right)^n(F(x_0)-F^\ast).
    \end{equation}
\end{theorem}
    Proof is in Appendix \ref{appendix non diff theorem 2}. It uses the same Lyapunov function as in Theorem \ref{theorem 1} and it makes use of the \ref{prox-grad} theorem from \cite{beck2017}.
\\
It might seem weird, for practical purpose, to ask strong quasar convexity at a point which is not a minimizer of the aforementioned function. 
%Actually, as we do not know $x^\ast_F$ in practice, there is no way to be sure that $f$ is strongly quasar convex at this point, except if the function is strongly convex, which cancels the interest of the result. However 
It allows to show that theoretically, convexity of $f$ is not necessary to extend the accelerated result in the differentiable case to the composite case. Note that together with difficulties of defining (strongly) quasar convex functions with respect to an arbitrary point, arguments used in our proof do not work if $\gamma < 1$. In that case, the $\gamma <1$ relaxation is non trivial, while in some cases it just demands computations adjusting. 
% \paragraph{Remark} As a direct consequence of Lemma \ref{lemma sqc not minimize}, if $F$ is such as it is defined in Theorem \ref{theorem 2}, it is then $(1,\mu)$-strongly quasar convex with respect to its minimizer $x^\ast_F$.
\section{Continuous analysis}\label{sec:continu}
Considering first order algorithms as a discretization of an ODE is a powerful tool, helping to gain intuition, and find and generalize new convergence results \cite{suboydcandes, siegel2021accelerated}.
The classic ODE version of NAG (Algorithm~\ref{algo}) is written as the following:
\begin{equation*}
    \ddot{X}(t) + \gamma(t)\dot{X}(t) + \nabla F(X(t)) = 0,
\end{equation*}
where $\gamma(t)$ is a non negative function. The solution of this ODE can be thought as the continuous analogous of the NAG algorithm. Using Lyapunov strategies, one can deduce convergence results on the solution of this system. This has been extensively done in the case of convex or strongly convex functions, \textit{e.g.} see
 \cite{suboydcandes,siegel2021accelerated, aujdossrondPL,hyppo, attouch2020firstorder,shi2018understanding,nesterov1983linear}.
 \subsection{High resolution ordinary differential equation}

In \cite{shi2018understanding} are introduced high resolution ODEs to study first order optimization algorithms. The interest of these ODEs is to describe more accurately the behaviour of the corresponding algorithms compared to the classical versions of these ODEs, thus renamed low-resolution ODEs. This is made by keeping $\bigO(\sqrt{s})$ terms during the discretization process, inducing that the ODE depends on the stepsize $s$ of NAG (Algorithm~\ref{algo}). 
\paragraph{$F$ $\mu$-strongly convex}
Considering a $\mu$-strongly convex function $F$,
Polyak's Heavy Ball \cite{POLYAK19641} and NAG-SC (see Section~\ref{sec:sc}) can be seen as discretizations of the same low resolution ODE:
\begin{equation*}
    \ddot{X}(t) + 2\sqrt{\mu} \dot{X}(t) + \nabla F(X(t)) = 0.
\end{equation*}
However, these algorithms correspond to two different high resolution ODEs (see \cite{shi2018understanding} for more details):
\begin{equation}\label{HB-EDO}\tag{HB-ODE}
    \ddot{X}(t) + 2\sqrt{\mu}\dot{X}(t) + (1 + \sqrt{\mu s})\nabla F(X(t)) = 0
\end{equation}
\begin{equation}\label{NAG-SC-ODE}\tag{NAG-SC-ODE}
    \ddot{X}(t) + 2\sqrt{\mu}\dot{X}(t) + \sqrt{s}\nabla^2 F(X(t))\dot{X}(t) + (1 + \sqrt{\mu s})\nabla F(X(t)) = 0
\end{equation}
As Polyak's Heavy ball cannot achieve an accelerated rate for smooth strongly convex functions, while NAG-SC can, understanding how the two high resolution ODEs differ is thus interesting. The only differing term is $\sqrt{s}\nabla^2 F(X(t))\dot{X}(t)$, which results in the corresponding algorithm to what the authors refer to as a gradient correction term. In (\ref{NAG-SC-ODE}), factorizing $\dot{X}$ terms, we see that the damping coefficient becomes $2\sqrt{\mu}+ \sqrt{s}\nabla^2 F(X(t))$, which is now adaptive to the position of $X$. In particular if $\dot{X}$ is colinear with the eigenvector corresponding to the highest eigenvalue of $\nabla^2 F$ (possibly $L$), the new damping rate increases, thus reducing oscillations.\\
In the $\mu$-strongly convex case, one can obtain the following convergence result for the solution of (\ref{NAG-SC-ODE}).
% As we stated in Proposition \ref{2 points scheme}, in the case of $(1,\mu)$-strongly quasar convex functions, Algorithm 1 with right parameters defines the same algorithm as the classical Nesterov Algorithm to minimize $\mu$-strongly convex and $L$-smooth functions. In that case, we will get the same high resolution ODE (\ref{NAG-SC-ODE}).
% We see then immediately that this Hessian term will induce specific behaviour in a non convex case. In particular if $\dot{X}$ is colinear to the eigenvector of $\nabla^2 F(X(t))$ associated with eigenvalue $-L$, the damping rate becomes $$(2\sqrt{\mu}-L \sqrt{s})\dot{X}(t).$$ Take $s = \frac{1}{L}$ and it is negative (if $\mu$ is not too close to $L$). The case of negative damping rate is unusual. In particular, it induces the increase of the total mechanical energy ($F(X(t)) -F^*+ \frac{1}{2}\lVert \dot{X}(t) \rVert^2$).
\begin{theorem}\label{thm:ode_sc}[Theorem~1, \cite{shi2018understanding}]
Let $F$ be $C^2$, $\mu$-strongly convex and $L$-smooth for some $0 < \mu \leqslant L$. Assume X is solution of (\ref{NAG-SC-ODE}) with $0< s \leqslant \frac{1}{L}$, $X(0) = X_0$ and $\dot{X}(0) =\frac{-2\sqrt{s}\nabla F(X_0)}{1+\sqrt{\mu s}}$. Then:
       \begin{equation}
            F(X(t))-F^\ast \leqslant  \frac{2 \lVert X_0 - x^\ast \rVert^2}{s} e^{-\frac{\sqrt{\mu}}{4}t}.
        \end{equation}
\end{theorem}
Theorem~\ref{thm:ode_sc} states that $\mu$-strong convexity allows for a linear convergence with rate $\bigO \left(-\sqrt{\mu}t\right)$. We extend this result when relaxing strong convexity with strong quasar convexity.
\paragraph{$F$ $(\gamma,\mu)$-strongly quasar convex} We first derive the high resolution ODE associated to Algorithm \ref{algo}.
Recall that with the choice of parameters of Theorem \ref{theorem 1}, Algorithm 1 is
 \begin{equation}\label{low res algo}\tag{NAG-SQC-DISCRETE}
    \left\{
    \begin{array}{ll}
        y_n = \frac{1}{1 + \sqrt{\mu s}} x_n + \frac{\sqrt{\mu s}}{1 + \sqrt{\mu s}}z_n \\
     x_{n+1} = y_n - s\nabla F(y_n)\\
        z_{n+1} = 1 - \gamma \sqrt{\mu s} z_n + \gamma \sqrt{\mu s}y_n - \sqrt{\frac{s}{\mu}} \nabla F(y_n)
    \end{array}
\right.
\end{equation}
The high resolution ODE associated with (\ref{low res algo})  can be written as follow:
% Using the Lyapunov function from \cite{shi2018understanding}, we see that we need to add a restriction on the lower curvature to get accelerated convergence, which is the same restriction as in theorem (). However, one can modify this function in order to get convergence without restriction.
\begin{equation}\label{NAG-SQC-ODE}\tag{NAG-SQC-ODE}
    \left(1+\frac{1-\gamma}{2}\sqrt{\mu s}\right)\ddot{X}(t) + (1+\gamma)\sqrt{\mu}\dot{X}(t) + \sqrt{s}\nabla^2 F(X(t))\dot{X}(t) + (1+\gamma\sqrt{\mu s})\nabla F(X(t)) = 0
\end{equation}
The derivation of this ODE is detailed in Appendix \ref{appendix high res}. Observe that taking $\gamma = 1$, we recover (\ref{NAG-SC-ODE}).
\paragraph{Remark}
The Hessian term will induce specific behaviour in a non convex case. If $f$ is $L$-smooth, and $\dot{X}$ is colinear to the eigenvector of $\nabla^2 F(X(t))$ associated with eigenvalue $-L$, the damping rate becomes $$((1+\gamma)\sqrt{\mu}-L \sqrt{s})\dot{X}(t).$$ Take $s = \frac{1}{L}$ and it is negative (if $\mu$ is not too close to $L$). The case of negative damping rate is unusual. In particular, it induces the increase of the total mechanical energy ($F(X(t)) -F^*+ \frac{1}{2}\lVert \dot{X}(t) \rVert^2$).
\\

Despite the eventual negative friction of (\ref{NAG-SQC-ODE}), one can get convergence results that are similar to the one obtained in the strongly convex setting \cite{siegel2021accelerated}, up to a $\gamma$ factor.
\begin{theorem}\label{prop edo haute res works}
Let $F$ be $C^2$, $(\gamma,\mu)-$strongly quasar convex and $L$-smooth for some $0 < \mu \leqslant L$, $\gamma \in (0,1]$. Assume X is solution of (\ref{NAG-SQC-ODE}) with $0 \leqslant s \leqslant \frac{1}{L}$, $X(0) = X_0$ and $\dot{X}(0) =0 $. Then:
       \begin{equation}
            F(X(t))-F^\ast \leqslant  K_0(\gamma,\mu,L,s)\frac{1}{\gamma}(F(X_0)-F^\ast)  e^{-\gamma\frac{\sqrt{\mu}}{2}t}
        \end{equation}
        where $ K_0(\gamma,\mu,L,s)$ can be uniformly bounded by $7$.
\end{theorem}
To prove Theorem \ref{prop edo haute res works}, we aim to show the linear decrease of a Lyapunov function of the form:
    \begin{equation}
        \mathcal{E}(t) = \delta(F(X(t))-F^\ast) + \frac{1}{2}\left\lVert \left(  1 + \frac{1-\gamma}{2}\sqrt{\mu s}\right)\dot{X}(t) + \lambda(X(t) - x^\ast) + \sqrt{s}\nabla F(X(t)) \right\rVert^2.
    \end{equation} 
where $\delta$, $\lambda$ belong in $\mathbb{R}$ and are well chosen parameters. See Appendix \ref{appendix high res} for the proof.\\
% We use almost the same Lyapunov function that the one we use in the low resolution case (), except a $\sqrt{s}\nabla F(X(t))$ term appears inside the norm. Our intuition is that this term compensate the slower decrease of $\dot{X}$ that occurs when the trajectory climb along a non convex part of the function (inducing negative friction), which make the gradient decrease. \\
% The $\sqrt{\mu}$ in the exponential exponent is how we characterize possibility of acceleration in the continuous case for strongly convex functions. The gradient flow (ODE version of gradient descent) achieves a similar rate of convergence, with a $\mu$ exponent instead of $\sqrt{\mu}$.\\
The result of Theorem \ref{prop edo haute res works} is quite surprising. It indicates that despite the occurring of negative damping along the trajectory, we can still achieve similar convergence bounds as for the strongly convex case. In particular, we do not need assumptions such as $(a,L)$-curvature (Definition~\ref{def:a_l_curvature}).
 Thus, while non convexity of strongly quasar convex functions may impact negatively the convergence of Algorithm \ref{algo}, it is not the case for the associated continuous system. This indicates that non convexity impacts the discretization process. We discuss this intuition in the next section.
\paragraph{Remark:} Taking $s = 0$ in (\ref{NAG-SQC-ODE}), we can automatically deduce the low resolution ODE associated with (\ref{low res algo}).

    \subsection{The continuous/discrete rupture}\label{section continuous discrete rupture}
As showed in the previous section, one can achieve accelerated convergence with the solution of the high resolution ODE (\ref{NAG-SQC-ODE}) associated to Algorithm 1, without the need to add curvature assumption.\\
As the Algorithm and the discrete Lyapunov function we use are discretization of continuous counterparts, one expects these to be similar. More precisely, ordering term by their dependence on $s$ ($1$, $\bigO(\sqrt{s}$), $\bigO(s),\dots)$, one expects the "main terms", namely the one with lower dependence on $s$, to appear both in continuous and discrete setting, while eventually, the discretization process will make appear new terms with higher dependency on $s$. In this section, we see why it is not necessarily the case when non convexity steps in.
 \paragraph{Continuous/discrete Lyapunov comparison}
Recall the continuous Lyapunov used to prove Theorem \ref{prop edo haute res works}.
\begin{equation}\label{section 4.3 continu Lyap}
        \mathcal{E}(t) = \delta(F(X(t))-F^\ast) + \frac{1}{2}\lVert \upsilon \dot{X}(t) + \lambda(X(t) - x^\ast) + \sqrt{s}\nabla F(X(t)) \rVert^2
    \end{equation}
    where, depending on the values of $\gamma, \mu, L$, we have $\upsilon \in [1,\frac{3}{2}], \delta \in [1,3]$ and $\lambda \in [\sqrt{\mu},\frac{9}{8}\sqrt{\mu}]$ (see the exact values in Appendix \ref{appendix high res}).
    The discrete Lyapunov used to prove Theorem \ref{theorem 1} can be written the following way:
\begin{equation}\label{section 4.3 discrete Lyap}
       E_n = F(x_n) - F^\ast +  \frac{1}{2}\lVert \frac{(y_n - y_{n-1})}{\sqrt{s}}  + \sqrt{\mu}(y_n- x^\ast)  + \sqrt{s} \nabla F(y_{n-1}) \rVert^2.
\end{equation}
The fact that $\{ E_n \}_{n \in \mathbb{N}}$ is a discretization of $\mathcal{E}$ appears here clearly.
\paragraph{Derivation difference}
 In the continuous case, we studied $\dot{\mathcal{E}}$, where we aimed to show $\dot{ \mathcal{E}}(t) \leqslant -\gamma \frac{\sqrt{\mu}}{2}\mathcal{E}(t)$. We emphasize the two following facts:
 \begin{itemize}
     \item Derivation of $\delta(F(X(t))-F^\ast)$ leads to $\delta \langle \nabla F(X),\dot{X} \rangle$.
     \item Derivation of the norm term leads, among other terms, to $-\langle \nabla F(X),\dot{X} \rangle$ (up to a parameter). 
 \end{itemize}
 Appropriate parameter tuning allows to cancel these terms, see proof of Theorem~\ref{prop edo haute res works}. 
 \\
In the discrete case, we study a discrete derivation $E_{n+1} - E_n$.  Importantly, the two aforementioned derivation, that were the same in the continuous setting, will be different here.
\begin{itemize}
    \item Discrete derivation of $ F(x_{n+1}) - F^\ast$ leads to $F(x_{n+1}) - F(x_n)$.
    \item Discrete derivation of the norm term brings, among other terms, to $-\langle y_n-x_n,\nabla F(y_n)  \rangle$.
\end{itemize}
More precisely, with slight modifications of the proof of Theorem \ref{theorem 1}, we get:
  \begin{align}
     E_{n+1} &\leqslant \underbrace{(1 -\gamma\sqrt{\mu s})E_n}_{\text{Aimed linear decrease}} + (1-\gamma\sqrt{\mu s})\underbrace{\left(F(y_n) - F(x_n) + \langle \nabla F(y_n),x_n - y_n \rangle \right)}_{\text{Derivation difference}}\nonumber \\
     & \underbrace{- \gamma (1-\gamma \sqrt{\mu s})\sqrt{\frac{\mu}{s}} \lVert y_n -x_n \rVert^2}_{\leq 0}.\label{deriv diff}
\end{align}
In (\ref{deriv diff}), apart from the negative kinetic term, that can be matched with a continuous counterpart ($ -\gamma\frac{\sqrt{\mu }\upsilon^2}{2}\lVert \dot{X}(t)\rVert^2$ term in (\ref{continuous kinetic energy})), it remains the difference of the two aforementioned derivation, \textit{i.e.}
\begin{equation}
    F(y_n) - F(x_n) + \langle \nabla F(y_n),x_n - y_n \rangle.
\end{equation}
In contrast to the continuous case, we can not just use parameter tuning to control this term.
Convexity is exactly the assumption we need to get rid of this derivation difference term. However, if $F$ is $(a,L)$-curvatured with $a < 0$, the best control we have is the following:
\begin{equation}
    F(y_n) - F(x_n) + \langle \nabla F(y_n),x_n - y_n \rangle \leqslant -\frac{a}{2}\lVert x_n-y_n \rVert^2.
\end{equation} The lower below zero $a$ is, or with other words the more non convexity we allow, the more this extra term can get large, up to a $+\frac{L}{2}\lVert x_n - y_n \rVert^2$ term in the $L$-smooth case ($a = -L$). If this term was appearing in the continuous case, the same restriction on the curvature as for Theorem \ref{theorem 1} would be necessary. This highlights that when considering properties satisfied by the solution of a continuous system, the transfer of these properties to the discrete case via discretization can be hurt by non convexity.
\section{Geometrical considerations}\label{section geometrical considerations}
In this section, we create a connection with the class of smooth PL functions (Definition \ref{PL}).
Then we present some new properties of strongly quasar convex functions, together with properties concerning more general non convex functions. 
\subsection{The frontier property allowing acceleration}
Recall that a function is $\mu$-Polyak Lojasiewicz ($\mu$-PL) if there exists $\mu>0$ such that for all $x \in \mathbb{R}^d$:
\begin{equation}\label{def PL}
    \frac{1}{2 \mu}\lVert \nabla F(x) \rVert^2 \geqslant F(x) - F^\ast.
\end{equation}
As already mentioned in section \ref{section strong conv et question of acceleration}, according to \cite{PLlowerbound}, for $\mu$-PL and $L$-smooth functions we can not get the acceleration phenomenon we witness for strongly convex functions. As we know that this acceleration can occur for the class of smooth strongly quasar convex function, it is interesting for a comprehension purpose to understand what is the gap between these functions and smooth PL functions. In other words:
\begin{center}
     \textbf{What is missing for smooth PL functions to obtain acceleration ?}
\end{center}
We propose an answer in this section.  Strong quasar convexity implies uniqueness of minimizer, so for the sake of comparison we will consider the class of smooth PL functions with a unique minimizer. Note that the function built in \cite{PLlowerbound} to get the lower bound on smooth PL functions has a unique minimizer. There is thus still a point for comparison with this restricted class of smooth PL functions with unique minimizers, as gradient descent remains optimal when restricting to this class.
\begin{theorem}\label{theorem 4}
    Suppose $F$ is a $\mu$-PL, $L$-smooth function with a unique minimizer for some $0<\mu \leqslant L$. There exists $(\gamma,\mu') \in (0,1] \times \mathbb{R}^\ast_+$ such that $F$ is ($\gamma$,$\mu'$)-strongly quasar convex if and only if there exists some $a > 0$ such that F is satisfying the following \textbf{uniform acute angle condition}:
           \begin{equation}\label{UAAC CONDITION}\tag{UAAC}
             \forall x \in \mathbb{R}^d, \quad 1 \geqslant \frac{\langle \nabla F(x),x - x^\ast \rangle}{\lVert  \nabla F(x) \rVert \lVert x - x^\ast \rVert} \geqslant a> 0.
         \end{equation}
\end{theorem}
The proof of Theorem~\ref{theorem 4} and complements are deferred in Appendix \ref{appendix PL}, in which we give explicit parameters ($\gamma,\mu')$ depending on $a$, $\mu$ and $L$. \\
The (\ref{UAAC CONDITION}) condition can be interpreted in the following way: for all $x \in \mathbb{R}^d$, the descent direction ($-\nabla F(x)$) forms an acute angle with vector starting from $x$ to the minimizer $x^\ast$. When it does not hold, following descent direction bring us to an orthogonal or opposite direction to the one that would make us closer to the minimizer.
\paragraph{Comments on Theorem \ref{theorem 4}}
  The (\ref{UAAC CONDITION}) condition states that the momentum we accumulate is coherent, as it is directed toward the minimizer.
     For $\mu$-PL and $L$-smooth functions, that are not necessarily satisfying the (\ref{UAAC CONDITION}) condition, we know that momentum does not allow to achieve accelerated rate \cite{PLlowerbound}. Worse, in this case momentum appears to hurt the convergence rate. While the Polyak's Heavy Ball algorithm also leads to a linear convergence that is not better than gradient descent \cite{danilova2020non}, it also deteriorate as we increase momentum. Thus, the fact that (\ref{UAAC CONDITION}) does not hold can make momentum hurt the convergence speed.

\subsection{New properties of strongly quasar convex functions}
\paragraph{How weaker than strong convexity is strong quasar convexity ?}
Compared with strong convexity, the main difference is that we lose lots of \textit{local information}. While for $C^2$ $\mu$-strongly convex functions we have $ \langle \nabla^2 F(x)y,y\rangle \geqslant \mu \lVert y \rVert^2$ for all $x,y \in \mathbb{R}^d$ (or equivalently, all eigenvalues of the Hessian matrix are above $\mu$), we lose this regularity with strongly quasar convex functions. Actually, we only have, on average, a similar regularity on the segment joining points $x\in \mathbb{R}^d$ and the minimizer $x^\ast$.
\begin{proposition}
    Let $F$ be $C^2$ and $(\gamma,\mu)$-strongly quasar convex for some $\gamma \in (0,1]$, $\mu > 0$. Let $x \neq x^\ast$ and $t >0$. Then:
    \begin{align}
\frac{1}{t} \int_0^t \frac{\langle \nabla^2 F(x^\ast + s(x-x^\ast))(x^\ast - x),x^\ast - x \rangle}{\lVert x-x^\ast \rVert^2} ds> \gamma \frac{\mu}{2} .
\end{align}
\end{proposition}See appendix \ref{appendix average strong quasar convexity} for the proof. 
\paragraph{Local strong convexity for $C^2$ functions}
One may think that the property of uniqueness of the minimizer of strongly quasar convex functions induces that when we are close enough to the minimizer, the function is bowl-shaped and we avoid negative curvature, \textit{i.e.} the function is locally convex around the minimizer. This is true if the function is $C^2$. The two following results holds under the assumption of a quadratic growth function with a unique minimizer, which is a weaker statement than strong quasar convexity (see Section \ref{quasar convex sction 2}).
\begin{proposition}
      Let $F$ be $C^2$, with a unique minimizer $x^\ast$, and a $\mu$-quadratic growth. Then, there exists $\eta >0$ such that for all $x \in B(x^\ast,\eta$), $F$ is strongly convex.
\end{proposition}
\begin{sproof}
    The quadratic growth property around $x^\ast$ implies that $\nabla^2 F(x^\ast)$ is definite positive. By continuity of the Hessian we get the result. Rigorous proof is deferred in appendix \ref{appendix hess lip}.
\end{sproof}
% If we add a Hessian Lipschitz property, we can uniformly bound the distance from the minimizer that guaranties we do not reach a certain negative curvature.
% \begin{proposition}\label{hessien lipchz}
%       Let $F: \mathbb{R}^d \to \mathbb{R}$ be $C^2$, with a unique minimizer $x^\ast$, with $\mu$-quadratic growth, and with its Hessian being $\rho$-Lipschitz. If for some $s \in \mathbb{R}$, we have $\lVert x-x^\ast \rVert\leq\frac{\mu - s}{\rho}$, then: 
%     \begin{equation}
%         \frac{(x-x^\ast)^T\nabla^2 F(x)(x-x^\ast)}{\lVert x-x^\ast \rVert^2} \geqslant s
%     \end{equation}
%     Equivalently, $\lVert x-x^\ast \rVert\leq\frac{\mu - s}{\rho}$ implies that all eigenvalues of $\nabla^2 F(x)$ are above $s$.
% \end{proposition}
% Proof is in Appendix \ref{appendix hess lip}.\\
\paragraph{Pathological behaviour for non $C^2$ functions}
% However, the fact that quasar convex hypothesis gives regularity between a point $x$ and the minimizer $x^\ast$, and not between two arbitrary points $x$ and $y$, lets think that it is unlikely to control negative curvature between any pair of points. Moreover, even when considering the segment joining $x^\ast$ to a point $x$, there is no local convexity property, as stated by the following result.
When dropping the $C^2$ assumption, one can not ensure the local convexity property anymore, even when considering segments joining the minimizer (on which we have the stronger structure assumption). This is stated in the following result.
\begin{proposition}\label{prop contre exemple}
    One can construct $f: [0,1] \to \mathbb{R}$ strongly quasar convex with minimizer $x^\ast$, $L$-smooth, such that for all $x \neq x^\ast$, there exists $x_0 \neq x^\ast$ such that:
    \begin{equation}
        \lvert x^\ast - x_0\rvert \leq\lvert x^\ast - x \rvert \text{ and } f''(x_0) = -L
    \end{equation}
\end{proposition}
\begin{sproof}
The idea behind the construction is the following: the function $f$ is build such that its curvature alternates $L$ and $-L$, given that the $L$ curvature will occur more often than $-L$ to ensure that the function grows enough to be strongly quasar convex.\\ 
To create such a function, we define the following set
\begin{equation}
    E:= \bigcup_{n\geqslant 1} \left( \underbrace{\left[\frac{1}{2^n}, \frac{1}{2^{n-1}} \right)}_{\text{partition of [0,1]}} \cap \underbrace{\left[\frac{1}{2^n} + \frac{3}{4}\frac{1}{2^n},\frac{1}{2^{n-1}} \right]}_{\text{subpart of partition}} \right) = \bigcup_{n\geqslant 1} \left[\frac{7}{4}\frac{1}{2^n}, \frac{1}{2^{n-1}} \right).
\end{equation}
Then we define $f$ on $[0,1]$, such that:
$$
f''(x) = \left\{
    \begin{array}{ll}
         -L &  \mbox{if } x \in E \\
         L & \mbox{else.}
    \end{array}
\right.
$$
with $f'(0)=f(0) = 0$. $f$ is clearly such that it will reach its lower curvature $-L$ in all vicinity of its minimizer. It remains to show that it is strongly quasar convex, which is done in Appendix \ref{appendix non convex example}.\\
\end{sproof}
% \paragraph{Other pathological examples}The construction used to show Proposition \ref{prop contre exemple} can be adapted to make other pathological behaviours. For example, consider $\mu$-quadratic growth functions, \textit{i.e.} which satisfy for all $x\in \mathbb{R}^d$, $f(x) - f^\ast \geqslant \frac{\mu}{2}\lVert x-x^\ast \rVert^2$. This defines functions lower bounded by a quadratic, but it may have critical points that are not (global) minima. However due to the local regularity around a global minimizer $x^\ast$ offered by the quadratic growth, one may ask the following question:
% \begin{center}
%   \textbf{Can we ensure that a smooth function with unique minimizer $x^\ast$ satisfying $\mu$-quadratic growth has no other critical points when sufficiently close to $x^\ast$ ?}
% \end{center}
% We address this question in appendix \ref{appendix non convex example}, answering negatively by constructing a counter example based on the aforementioned construction.

    \section{Numerical experiments details}\label{appendix numerical}
\subsection{Algorithm performance}
We tested our algorithm on a function $h: \mathbb{R}^d \to \mathbb{R}$, $d = 100$, where 
\begin{equation}
    h(x) = f(\lVert x \rVert)g\left(\frac{x}{\lVert x \rVert} \right)
\end{equation}
with $f(t) = t^2$, and 
\begin{equation}
    g(x_1,\dots,x_{100}) = \sum_{i = 1}^{100} \left(a_i \sin(b_i x_i)^2 \right) + 1
\end{equation}
where $\{ a_i \}_i$ are independently and uniformly distributed on $[0,1]$ and the $\{ b_i \}_i$ are independently and uniformly distributed on $[-2.5,2.5]$. This function is $(1,2)$-strongly quasar convex, as $f$ is $(1,2)$-strongly quasar convex (see Proposition \ref{prop h is qsc}). Recall that this type of functions exhibits, by construction, lots of negative curvature. Hence, importantly, on these functions the assumption on the curvature of Theorem \ref{theorem 1} is not satisfied. 

% In particular these functions are not $C^2$ at their minimizer $0$, and one can show that there is no local convexity phenomenon for this function. 
\begin{figure}
    \centering
    
    \includegraphics[scale=0.425]{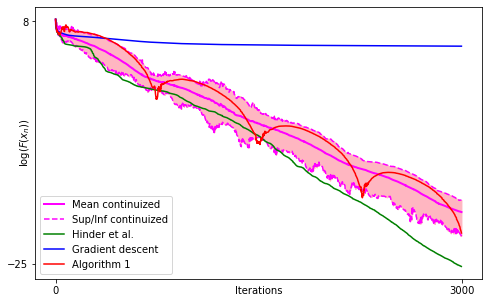}
        \includegraphics[scale=0.425]{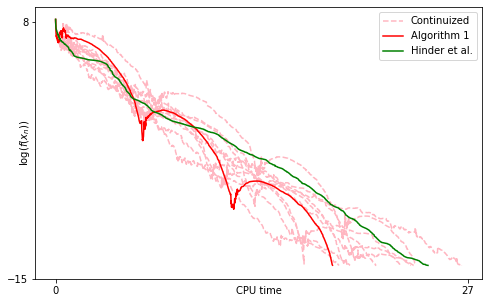}
        \includegraphics[scale=0.425]{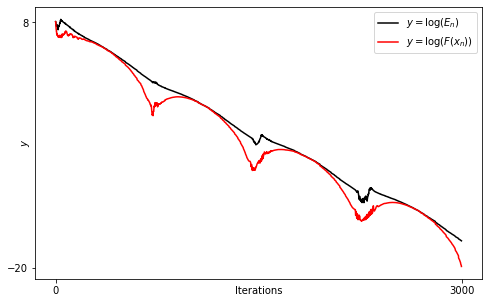}
    \caption{We compare the performance of an Algorithm using a line search procedure (Hinder et al. \cite{hinder2023nearoptimal}), a stochastic algorithm (continuized \cite{even2021continuized}), and Algorithm \ref{algo}. It is done iteration wise on the top left plot, while the top right compare the time needed to achieve a $\varepsilon$-solution. On the lowest plot we show the behaviour of Algorithm \ref{algo} with our choice of parameter in the presence of strong negative curvature regions.}
    \label{fig:2}
\end{figure}
\paragraph{Description of the experiments}
We performed numerical experiments on this function. We computed $L$ at each iteration with the same backtracking process as in \cite{hinder2023nearoptimal}. We compared the performance of Algorithm 1 using our choice of parameters with two other methods:
\begin{enumerate}
    \item Algorithm 1 with line search computing the $(\alpha_n)_n$ \cite{hinder2023nearoptimal}.
    \item Algorithm 1 with stochastic coefficients obtained using continuized framework \cite{even2021continuized, wang2023continuized}.
    \end{enumerate} As the line search procedure induces more computational complexity, we compared the performance in two different ways: firstly iteration wise, and secondly we compared the CPU time needed to achieve an $\varepsilon$-precision, \textit{i.e.} a point $\hat{x}$ such that $h(\hat{x}) - h^\ast < \varepsilon$, $\varepsilon > 0$.
    \\
    The continuized framework leads to stochastic algorithm and thus to result of convergence are of stochastic nature. Hence, to make the comparison with deterministic algorithms more relevant, we ran several times the algorithm. We make the following precision:
    \begin{itemize}
        \item For the iteration wise comparison, displayed on top left of figure \ref{fig:2}, we ran 50 times the algorithm. We plot the mean trajectory (Mean continuized), as well as the infimum and supremum of all the trajectories along the iterations (Sup/Inf continuized). The pink zone between these two plots thus contains all the trajectories.
        \item For the CPU time needed to attain a $\varepsilon$-precision, displayed on top right of figure \ref{fig:2}, we ran 10 times the algorithm. Here we simply plot 10 trajectories, corresponding each to a different run of the algorithm. We set $\varepsilon = 10^{-6}$.
    \end{itemize}
\paragraph{A note on the using of Backtracking} When using backtracking,
%there is necessarily a divergence with our theoretical background. The 
we compute $L_n$ such that $F(y_n - \frac{1}{L_n}\nabla F(y_n)) \leqslant F(y_n) - \frac{1}{2 L_n}\lVert  
\nabla F(y_n))  \rVert^2$. To compute $y_n$, we need $\alpha_n$, which also needs an estimate on $L$, for which we chose the previous $L_{n-1}$. Our theoretical framework does not consider the use of backtracking, and it is of interest whether our guarantees can be extended when using this procedure. 
\paragraph{Observations}
This precision being made, we can now state our observations.

\begin{enumerate}
   
    \item The two top plots displayed in figure \ref{fig:2} are not very surprising: iteration wise, the binary search procedure offers a better speed. However when considering the CPU time needed to achieve a $\varepsilon$-precision, doing without line search allows for better performance as the iterations are less computationally heavy. 
     \item We observe empirically that a high amount of strong negative curvature encountered during the running of Algorithm 1 correlate with "bad behaviour" of the algorithm. The lowest display of Figure \ref{fig:2} is characteristic of the non monotone decreasing behaviour of the Lyapunov function $E_n$ we can witness when the algorithm crosses strong negative curvature regions.
\end{enumerate}

\subsection{Explicit expression of the function displayed in Figure \ref{fig:1}}
The function displayed in Figure \ref{fig:1} is $h: \mathbb{R}^2 \to \mathbb{R}$, and is built in the following way:
\begin{equation}
    h(x) = f(\lVert x \rVert)g\left(\frac{x}{\lVert x \rVert} \right)
\end{equation}
where $f(t) = t^2$ and 
\begin{equation}
    g(x_1,x_2) = \frac{1}{4N} \sum_{i=1}^N \left(a_i \sin(b_i x_1)^2 + c_i \cos(d_i x_2)^2 \right) + 1
\end{equation}
with $N = 10$ and the $\{ a_i \}_i$, $\{ c_i \}_i$ are independently and uniformly distributed on $[0,20]$, and the $\{ b_i \}_i$, $\{ d_i \}_i$ are independently and uniformly distributed on $[-25,25]$. 

% \section{Conclusion}
% In this paper, we highlight that the Nesterov accelerated gradient algorithm may need curvature assumption to get accelerated rate in a non convex setting. As observed in previous works, we saw here that too strong negative curvature is difficult for this algorithm, at least regarding the nature of convergence results we are seeking. Interestingly, as it is the case for Polyak's Heavy Ball \cite{POLYAK19641} in the strongly convex case, we saw that proving accelerated convergence of high resolution ODE associated to NAG in (strongly) quasar convex does not ensure convergence of discrete counterpart. Finally, it is still an open question whether there exists a deterministic algorithm achieving an accelerated rate on the class of smooth (strongly) quasar convex functions, without adding assumption, and without a subroutine to compute a parameter (as binary line search).
\section*{Acknowledgements}
This work was supported by PEPR PDE-AI, the ANR MICROBLIND (grant ANR-21-CE48-0008), the ANR Masdol (grant ANR-19-CE23-0017), and the ANR SOS2ID (grant ANR-24-CE40-3786).
\bibliographystyle{plain} 
\bibliography{bib}

\appendix

\section{Some properties on strongly quasar convex functions}
In this section, we prove our claims about properties of strong quasar convex functions.
\subsection{Strong convexity on average on segments joining the minimizer}\label{appendix average strong quasar convexity}
\begin{proposition*}
      Let $F$ be $(\gamma,\mu)$-strongly quasar convex. Let $x \neq x^\ast$ and $t >0$. Then we have
    \begin{align}
\frac{1}{t} \int_0^t \frac{\langle \nabla^2 F(x^\ast + s(x-x^\ast))(x^\ast - x),x^\ast - x \rangle}{\lVert x-x^\ast \rVert^2} ds \geqslant \gamma \frac{\mu}{2} 
\end{align}
\end{proposition*}
\begin{proof}
    Let $F: \mathbb{R}^d \to \mathbb{R}$ $(\gamma,\mu)$-strongly quasar convex. Define, for $t \in \mathbb{R}_+$, the function $g(t) = F(x^\ast + t(x-x^\ast))$. We have $g'(t) = \langle \nabla F(x^\ast + t(x-x^\ast)),x-x^\ast \rangle$. By strong quasar convexity of $F$, we have
\begin{align}
    &F(x^\ast + t(x-x^\ast)) + \frac{1}{\gamma}\langle \nabla F(x^\ast + t(x-x^\ast), x^\ast - (x^\ast + t(x-x^\ast) \rangle + \frac{\mu}{2}\lVert x^\ast - (x^\ast + t(x-x^\ast) \rVert^2 \leqslant F^\ast \\
    &\Rightarrow g(t) - \frac{t}{\gamma}g'(t)+ \frac{\mu t^2}{2}\lVert x -x^\ast \rVert^2  \leqslant g(0) \\
    &\Rightarrow \gamma\frac{g(t) - g(0)}{t}+ \frac{\gamma \mu t}{2}\lVert x -x^\ast \rVert^2   \leqslant g'(t) = \int_0^t g''(s)ds = \int_0^t \langle \nabla^2 F(x^\ast + s(x-x^\ast))x^\ast - x,x^\ast - x \rangle ds
\end{align}
Where for the last line we suppose $t>0$. From this we deduce:
\begin{align}
 \frac{1}{t}\int_0^t \frac{\langle \nabla^2 F(x^\ast + s(x-x^\ast))(x^\ast - x),x^\ast - x \rangle}{\lVert x-x^\ast \rVert^2} ds \geqslant \gamma \frac{\mu}{2}
\end{align}
\end{proof}
In particular, the above reasoning remains true for non strongly quasar convex functions taking $\mu = 0$, inducing a on average convexity on segments joining minimizers:
\begin{align}
 \frac{1}{t}\int_0^t \frac{\langle \nabla^2 F(x^\ast + s(x-x^\ast))(x^\ast - x),x^\ast - x \rangle}{\lVert x-x^\ast \rVert^2} ds \geq0
\end{align}
\subsection{Synthetic strongly quasar convex example proof}\label{appendix synthetic}
% In this section, we will use the useful following characterization of strong quasar convexity.
% \begin{lemma}[\cite{hinder2023nearoptimal}, lemma 11]\label{strongly convex characterization}
%     Let $f: \mathcal{X} \to \mathbb{R}$ be differentiable function with a minimizer $x^\ast$, where the domain $\mathcal{X} \subset \mathbb{R}^d$ is open and convex. Then, the following two statements:
%     \begin{equation}\label{strongly quasar conv non diff}
%         f(tx^\ast + (1-t)x) + t\left(1-\frac{t}{2-\gamma} \right)\frac{\gamma \mu}{2}\lVert x^\ast - x \rVert^2 \leqslant \gamma t f(x^\ast) + (1-\gamma t)f(x),\forall x\in \mathcal{X}, t \in [0,1]
%     \end{equation}
%     \begin{equation}\label{strongly quasar conv def}
%         f(x^\ast) \geqslant f(x) + \frac{1}{\gamma} \langle \nabla f(x),x^\ast-x \rangle + \frac{\mu}{2} \lVert x-x^\ast \rVert^2, \forall x \in \mathcal{X}
%     \end{equation}
%     are equivalent for all $\mu \geqslant 0$, $\gamma \in ]0,1]$.
% \end{lemma}
Let $f: \mathbb{R} \to \mathbb{R}$ such that it is $(\gamma,\mu)$-strongly quasar convex and $f(0) = 0 = f^\ast$. Let $g: S^{d-1} \to \mathbb{R}$ differentiable and such that $g(x)\geqslant 1 $ for all $x \in S^{d-1}$. We define 
\begin{equation}
    h(x) = f(\lVert x \rVert)g\left(\frac{x}{\lVert x \rVert} \right)
\end{equation}
We have $h(x) \to 0$ as $x \to 0$, we extend $h$ to $0$ by continuity defining $h(0) = 0$. We clearly have $h^\ast = h(0) = 0$.
\begin{proposition}\label{prop h is qsc}
    $h$ is $(\gamma,\mu)$-strongly quasar convex.
\end{proposition}
\begin{proof}
We have
\begin{equation*}
    h(x) - f(\lVert x \rVert) = f(\lVert x \rVert)(g\left(\frac{x}{\lVert x \rVert} \right)-1).
\end{equation*}
As $g\left(\frac{x}{\lVert x \rVert} \right)-1$ is nonnegative and $C^1$, following Appendix D.3 \cite{hinder2023nearoptimal}, we have that $h(x) - f(\lVert x \rVert)$ is $\gamma$-quasar convex. Thus $h(x) = h(x) - f(\lVert x \rVert) +  f(\lVert x \rVert)$ is a sum of a $\gamma$-quasar convex and a $\mu$-strongly quasar convex function, which is $\mu$-strongly quasar convex (see \cite{hinder2023nearoptimal} Appendix D.3).

% We use lemma \ref{strongly convex characterization} characterization of quasar strong convexity. We thus aim to show that for all $x \in \mathbb{R}^d$ and all $t \in [0,1]$, we have
% \begin{equation}
%         h(tx^\ast + (1-t)x) + t\left(1-\frac{t}{2-\gamma} \right)\frac{\gamma \mu}{2}\lVert x^\ast - x \rVert^2 \leqslant \gamma t h(x^\ast) + (1-\gamma t)h(x)
%     \end{equation}
%     First since $x^\ast = 0$ we have
%     \begin{align}
%     h(tx^\ast + (1-t)x) = h((1-t)x) = f((1-t)\lVert x \rVert)~g\left(\frac{x}{\lVert x \rVert} \right)
% \end{align}
% By strongly quasar convexity of $f$, we get 
% \begin{align}
%     h((1-t)x) &\leqslant \left( (1-\gamma t)f(\lVert x \rVert) - t\left(1-\frac{t}{2\gamma} \right)\frac{\mu \gamma}{2}\lVert x^\ast - x \rVert^2\right)g\left(\frac{x}{\lVert x \rVert} \right)\\
%     &=(1-\gamma t)\underbrace{f(\lVert x \rVert)g\left(\frac{x}{\lVert x \rVert} \right)}_{= h(x)}  - t\left(1-\frac{t}{2\gamma} \right)\frac{\mu \gamma}{2}\lVert x^\ast - x \rVert^2 g\left(\frac{x}{\lVert x \rVert} \right)
% \end{align}
% We conclude by computing:
% \begin{align}
%     h((1-t)x) +  t\left(1-\frac{t}{2\gamma} \right)\frac{\mu \gamma}{2}\lVert x^\ast - x\rVert^2 \leqslant (1-\gamma t)h(x) +t\left(1-\frac{t}{2\gamma} \right)\frac{\mu \gamma}{2}\lVert x^\ast - x \rVert^2\left( 1- g\left(\frac{x}{\lVert x \rVert} \right) \right)
% \end{align}
% Recalling $h^\ast = h(0) = 0$, we conclude using our condition $g\geq1$, and we do have the characterization of ($\gamma,\mu)$-strong quasar convexity.
\end{proof}
\subsection{Local strong convexity around the minimizer for $C^2$ functions}\label{appendix hess lip}
\begin{proposition*}
    Let $F$ be $C^2$, with a unique minimizer $x^\ast$, with $\mu$-quadratic growth. There exists $\eta >0$ such that for all $x \in B(x^\ast,\eta$), $F$ is strongly convex.
\end{proposition*}
\begin{proof}
\textbf{Step 1: $\nabla^2 F(x^\ast)$ is definite positive}
    We start by showing that $\nabla^2 F(x^\ast)$ is definite positive, \textit{i.e.} for all $x\in \mathbb{R}^d \backslash \{ 0 \}$ we have $\langle \nabla^2 F(x^\ast)x,x\rangle>0$.
    Let $g(h) = F(x^\ast + h(x-x^\ast))$, $h\geq0$ and $x \neq x^\ast$. We perform an order 2 Taylor development at $0$ of $g$:
    \begin{align}
        &g(h) = g(0) + hg'(0) + \frac{h^2}{2}g''(0) + o(h^2)\\
        \Leftrightarrow &F(x^\ast + h(x-x^\ast) = F^\ast + h\underbrace{\langle \nabla F(x^\ast),x-x^\ast \rangle}_{=0} + \frac{h^2}{2}\langle \nabla^2 F(x^\ast)(x-x^\ast),x-x^\ast \rangle + o(h^2)\\
        \Leftrightarrow &F(x^\ast + h(x-x^\ast) - F^\ast = \frac{h^2}{2}\langle \nabla^2 F(x^\ast)(x-x^\ast),x-x^\ast \rangle + o(h^2)
    \end{align}
  As $F$ is with $\mu$-quadratic growth, we have:
\begin{equation}
    F(x) - F^\ast \geqslant \frac{\mu}{2}\lVert x - x^\ast \rVert^2
\end{equation}
We thus have
\begin{align}
    &\frac{h^2}{2}\langle \nabla^2 F(x^\ast)(x-x^\ast),x-x^\ast \rangle + o(h^2) \geqslant \frac{\mu h^2 }{2}\lVert x - x^\ast \rVert^2 \\
    \Rightarrow &\frac{\langle \nabla^2 F(x^\ast)(x-x^\ast),x-x^\ast \rangle}{\lVert x - x^\ast \rVert^2} + \frac{2}{\lVert x-x^\ast \rVert^2}\frac{o(h^2)}{h^2} \geqslant \mu
\end{align}
Taking $h \to 0$, we get 
\begin{align}\label{mu-defined positive}
    \frac{\langle \nabla^2 F(x^\ast)(x-x^\ast),x-x^\ast \rangle}{\lVert x - x^\ast \rVert^2} \geqslant \mu
\end{align}
Taking $x = y+x^\ast$, we get that for all $y \in \mathbb{R}^d \backslash \{0\}$, we have $\langle \nabla^2 F(x^\ast)y,y \rangle\geqslant \mu\lVert y \rVert^2 > 0$. We showed that $\nabla^2 F(x^\ast)$ is definite positive.
\textbf{Step 2: extension in a local vicinity}
We showed in step 1 that all eigenvalues of the Hessian matrix evaluated at $x^\ast$ are strictly positive. As we assumed $\nabla^2 F$ is continuous, for all $\varepsilon$ such that $0 < \varepsilon < \mu$, there exists $\eta >0$ such that for all $x \in B(x^\ast,\eta)$, the eigenvalues of $\nabla^2 F(x)$ are above $\mu-\varepsilon$. This means that on this ball $F$ is strongly convex.
\end{proof}
\subsection{No local convexity for non $C^2$ functions: a non convex pathological constructions}\label{appendix non convex example}
We construct a function that exhibit pathological non convex behaviour around their minimizer. This function is defined on $\mathbb{R}$, so this pathological behaviour does not need several dimensions to happen.
\subsubsection{Proof of proposition \ref{prop contre exemple}}
\begin{proposition*}
    One can construct $f: [0,1] \to \mathbb{R}$ strongly quasar convex with minimizer $x^\ast$, $L$-smooth, such that for all $x \neq x^\ast$, there exists $x_0 \neq x^\ast$ such that:
    \begin{equation}
        \lvert x^\ast - x_0\rvert \leq\lvert x^\ast - x \rvert \text{ and } f''(x_0) = -L
    \end{equation}
\end{proposition*}
\begin{proof}
Let 
\begin{equation}
    E:= \bigcup_{n\geqslant 1} \left( \underbrace{\left[\frac{1}{2^n}, \frac{1}{2^{n-1}} \right[}_{\text{partition of [0,1]}} \cap \underbrace{\left[\frac{1}{2^n} + \frac{3}{4}\frac{1}{2^n},\frac{1}{2^{n-1}} \right]}_{\text{subpart of partition}} \right) = \bigcup_{n\geqslant 1} \left[\frac{7}{4}\frac{1}{2^n}, \frac{1}{2^{n-1}} \right[ 
\end{equation}
and 
\begin{equation}
    E_n:= \bigcup_{k \geqslant n+1} \left[\frac{7}{4}\frac{1}{2^k}, \frac{1}{2^{k-1}} \right[ 
\end{equation}
Let $f$ be a function defined on $[0,1]$, such that:
$$
f''(x) = \left\{
    \begin{array}{ll}
         -L &  \mbox{if } x \in E \\
         L & \mbox{else.}
    \end{array}
\right.
$$
We suppose $f'(0)=f(0) = 0$. $f$ is clearly such that it will reach its lower curvature $-L$ in all vicinity of its minimizer. We now want to show that $f$ is strongly quasar convex.\\
Suppose $x$ is such that $\exists k\geqslant 1$, $x = \frac{1}{2^k}$. Then,
\begin{align}
    f'(x) = f'(x) - f'(0) = \int_{[0,x]} f''(s)ds = \int_{E_n} f''(s)ds + \int_{[0,x]\setminus E_n} f''(s)ds
\end{align}
By definition of $f''$, this simply becomes 
\begin{align}
    f'(x) = -L\lambda(E_n ) + L\lambda([0,x]\setminus E_n) = L(x - 2 \lambda(E_n))
\end{align}
Where $\lambda(.)$ is the Lebesgue measure. But by construction $\lambda(E_n) = \frac{1}{4}x$, which means 
\begin{equation}
    f'(x) = \frac{L}{2}x
\end{equation}
Now suppose $ x = \frac{1}{2^k} + \varepsilon$, where $k\geqslant 1 $ and $0<\varepsilon \leq\frac{3}{4}\frac{1}{2^k} $. This time we get 
\begin{align}
     f'(x) =  \int_{[0,x]} f''(s)ds =  \int_{[0,\frac{1}{2^k}]} f''(s)ds + \int_{[\frac{1}{2^k},x]} f''(s)ds = \frac{L}{2}\frac{1}{2^k} + L\varepsilon \geqslant \frac{L}{2}(\frac{1}{2^k} + \varepsilon) = \frac{L}{2}x
\end{align}
Finally, suppose $x = \frac{1}{2^n} + \frac{3}{4}\frac{1}{2^n} + \varepsilon $, where $n \geqslant 1$ and $0<\varepsilon < \frac{1}{4}\frac{1}{2^n}$.
\begin{align}
     f'(x) =  \int_{[0,x]} f''(s)ds =  \int_{[0,\frac{1}{2^n} + \frac{3}{4}\frac{1}{2^n}]} f''(s)ds + \int_{[\frac{1}{2^n} + \frac{3}{4}\frac{1}{2^n},x]} f''(s)ds = \frac{L}{2}\frac{1}{2^n} +  \frac{3L}{4}\frac{1}{2^n} - L\varepsilon \geqslant \frac{L}{2^n}
\end{align}
But here $x \leqslant \frac{1}{2^{n-1}}$, which gives 
\begin{align}
    f'(x) \geqslant \frac{L}{2} \frac{1}{2^{n-1}} = \frac{L}{2}x
\end{align}
Finally, for all $x \in [0,1]$, we have $f'(x) \geqslant \frac{L}{2}x$. To prove that this function is strongly quasar convex, remark by definition of $f''$ that $f(x) \leqslant \frac{L}{2}x^2$, and using $f'(x) \geqslant \frac{L}{2}x$, we have that for any $\mu > 0$:
\begin{align}
    f(x)-\frac{\mu+L}{L}f'(x)x + \frac{\mu}{2}x^2 \leqslant \frac{L}{2}x^2-\frac{\mu+L}{2}x^2 + \frac{\mu}{2}x^2=0 = f^\ast
\end{align}
In words we showed that $f$ is $(\frac{L}{\mu + L},\mu)-$strongly quasar convex. 
% In conclusion, we created a 1d function which is strongly quasar convex, and for which there exists no neighbourhood around minimizer such that negative curvature is excluded.
\end{proof}
\section{A PL-Strongly quasar convex link}\label{appendix PL}
 In this section we characterize the difference between smooth PL functions (Definitions \ref{def L-smooth} and \ref{definition PL} respectively) and smooth, strongly quasar convex functions. We use relation of $\mu$-PL functions with intermediate conditions that we can relate to strongly quasar convex functions.

% \paragraph{Remark 1} We can not claim that all the following lemmas are new. We indicated a citation when we were aware that a result already exists in the literature.

\paragraph{Remark 1}
We will introduce geometrical conditions that can hold considering projection onto the set of minimizers. As we aim to show a link with strong quasar convexity, we will restrict ourselves to functions with a unique minimizer. This means that some of the definitions we will introduce are specifically here restricted to this unique minimizer case.
\\
\\
For the first lemma, we introduce the following condition.
\begin{definition}[Error Bound]\label{EB}
$F: \mathbb{R}^d \mapsto \mathbb{R}$ is $\theta$-Error Bound ($\theta$-$EB$) if
\begin{align*}
   \forall x \in \mathbb{R}^d, \quad \lVert \nabla F(x) \rVert \geqslant \theta \lVert x- x^\ast \rVert
\end{align*}
\end{definition}
\begin{lemma}\label{lemma PL-EB}
    A $\mu$-PL function is $\mu$-EB. A $\theta$-EB and $L$-smooth function is $\frac{\theta}{L}$-PL.
\end{lemma}
 It is shown in \cite{karimi2020linear}.
%  As it is short we will give the proof.
% \begin{proof}
% \item \paragraph{PL $\Rightarrow$ EB}
% Let $F$ be $\mu$-PL.
%   By definition and using the fact that a $\mu$-PL function is also $\mu$-quadratic growth (see \cite{garrigos2023square} theorem 11, or \cite{karimi2020linear}):
%   \begin{align}
%       \lVert \nabla F(x) \rVert^2 \frac{1}{2 \mu} \geqslant F(x) - F^\ast \geqslant \frac{\mu}{2}\lVert x - x^\ast \rVert^2 \Rightarrow \lVert \nabla F(x) \rVert \geqslant \mu \lVert x - x^\ast \rVert
%   \end{align}
%   \item \paragraph{EB + $L$-smooth $\Rightarrow$ PL}
%   Suppose $F$ is $\mu$-EB and $L$-smooth. We have
%   \begin{align}
%       \lVert \nabla F(x) \rVert \overset{(EB)}{\geq} \theta \lVert x- x^\ast \rVert \overset{(\text{L-smooth})}{\geq} \frac{2 \theta }{L}(F(x)-F^\ast)
%   \end{align}
%   Where we use the fact that a $L$-smooth function verifies $F(x)-F^\ast \leqslant \frac{L}{2}\lVert x-x^\ast \rVert^2$.
% \end{proof}
We now introduce the notion of RSI functions \cite{RSIorigin}.
\begin{definition}\label{defRSI}
$F: \mathbb{R}^d \mapsto \mathbb{R}$ is $\nu$-RSI if
\begin{align*}
    \langle \nabla F(x), x - x^\ast \rangle \geqslant \nu \lVert x - x^\ast \rVert^2,\quad \forall x \in \mathbb{R}^d
\end{align*}
\end{definition}
Using Cauchy Schwartz, we immediately see that $\nu$-RSI implies $\nu$-EB.
We show that up to the a supplementary condition, the converse also holds.
    \begin{lemma}\label{lemma EB-RSI}
         Suppose $F$ is satisfying the following \textbf{uniform acute angle condition}:
         \begin{equation}\label{hyp1}\tag{UAAC}
             \forall x \in \mathbb{R}^d, \quad 1 \geqslant \frac{\langle \nabla F(x),x - x^\ast \rangle}{\lVert  \nabla F(x) \rVert \lVert x - x^\ast \rVert} \geqslant a > 0
         \end{equation}
         then
         \begin{equation*}
             F \text{ is }\theta\text{-}EB \Rightarrow F \text{ is }\theta a\text{-}RSI
         \end{equation*}
              \end{lemma}
         \begin{proof}
             We have
             \begin{align}
                  \theta \lVert x - x^\ast \rVert \overset{\theta-(EB)}{\leq} \lVert \nabla F(x) \rVert \overset{(\ref{hyp1})}{\leq} \frac{\langle \nabla F(x),x - x^\ast \rangle}{a \lVert x - x^\ast \rVert} \Rightarrow \theta a \lVert x - x^\ast \rVert^2 \leqslant \langle \nabla F(x),x - x^\ast \rangle
             \end{align}
         \end{proof}
The following is the last intermediate result.
\begin{lemma}\label{lemma QSC-RSI}
    Let $F$ be $L$-smooth. Then:
    \begin{enumerate}
        \item ($F$ is $(\gamma,\mu)$-strongly quasar convex) $\Rightarrow$ ($F$ is  $\frac{\gamma \mu}{2-\gamma}$-RSI)
        \item ($F$ is $\nu$-RSI) + ($\gamma < \frac{2\nu}{L}$) $\Rightarrow$ ($F$ is $\left(\gamma,\frac{\nu}{\gamma} - \frac{L}{2} \right)$-strongly quasar convex)
    \end{enumerate} 
\end{lemma}
\begin{proof}
    \item 
    \paragraph{Point 1.}
 Using definition of strong quasar convexity, and the fact that it implies $\frac{\mu \gamma}{2-\gamma}$ quadratic growth (Proposition \ref{SQC implis PL & QG}, 2.), we have:
    \begin{align}
         \langle \nabla F(x),x - x^\ast \rangle \geqslant \gamma( F-F^\ast) + \frac{\gamma  \mu}{2}\lVert x-x^\ast \rVert^2 &\geqslant \frac{\gamma^2 \mu}{2(2-\gamma)}\lVert x-x^\ast \rVert^2 +\frac{\gamma \mu}{2} \lVert x-x^\ast \rVert^2\\
         &= \frac{\gamma \mu}{2-\gamma} \lVert x-x^\ast \rVert^2
    \end{align}
    \item 
    \paragraph{Point 2.}We start with the definition of $\nu$-RSI:
\begin{align}
    \langle \nabla F(x), x-x^\ast  \rangle \geqslant \nu \lVert x-x^\ast \rVert^2 &\Rightarrow 0 \geqslant \frac{1}{\gamma} \langle \nabla F(x),x^\ast - x \rangle +\frac{\nu}{\gamma} \lVert x-x^\ast \rVert^2 \\
    &\Rightarrow -\frac{L}{2}\lVert x-x^\ast \rVert^2 \geqslant \frac{1}{\gamma} \langle \nabla F(x),x^\ast - x \rangle +\left(\frac{\nu}{\gamma} - \frac{L}{2} \right)\lVert x-x^\ast \rVert^2
\end{align}
Where $\gamma \in (0,1]$ is to be precised.
$L$-smooth property implies $F(x)-F^\ast \leqslant \frac{L}{2}\lVert x-x^\ast \rVert^2$, thus we have:
\begin{align}
    F^\ast \geqslant F(x) +\frac{1}{\gamma} \langle \nabla F(x),x^\ast - x \rangle +\left(\frac{\nu}{\gamma} - \frac{L}{2} \right)\lVert x-x^\ast \rVert^2
\end{align}
Hence, choosing $\gamma$ such that $\frac{\nu}{\gamma} - \frac{L}{2} >0\Rightarrow \gamma < \frac{2\nu}{L}$, we have that $F$ is $\left(\gamma,\frac{\nu}{\gamma} - \frac{L}{2} \right)$-strongly quasar convex.
\end{proof}
Note that we did not really need gradient Lipschitz property to hold, rather a weaker upper quadratic growth condition  (this also holds for Lemma \ref{lemma PL-EB}):
\begin{equation}
    \frac{L_0}{2}\lVert x - \hat{x} \rVert^2 \geqslant F(x) - F^\ast
\end{equation}
Importantly, $L_0$ may be significantly lower that $L$. See \cite{guilleescuret2020study} for a discussion about upper conditions.\\

With all these lemmas, we are ready to state our result.
\begin{theorem}\label{theorem PL}
    Let $F$ be $L$-smooth and $\mu$-PL with a unique minimizer. Then we have that there exists $\gamma,\mu' > 0$ such that $F$ is $(\gamma,\mu')$-strongly quasar convex if and only if for some $a \in (0,1]$, $F$ satisfies the following \textbf{uniform acute angle condition}:
         \begin{equation}\label{hyp}\tag{UAAC}
             \forall x \in \mathbb{R}^d, \quad 1 \geqslant \frac{\langle \nabla F(x),x - x^\ast \rangle}{\lVert  \nabla F(x) \rVert \lVert x - x^\ast \rVert} \geqslant a > 0
         \end{equation} 
     In particular if (\ref{hyp}) holds, then as long as $\gamma < \frac{2 \mu a}{L}$, $F$ is $(\gamma, \frac{\mu a}{\gamma} - \frac{L}{2})$-strongly quasar convex.
\end{theorem}
\begin{proof}
Let $F$ be $L$-smooth and $\mu$-PL. We have by Lemma \ref{lemma PL-EB} that $F$ is $\mu$-EB. By Lemma \ref{lemma EB-RSI} $F$ is $\mu a$-RSI.
By Lemma \ref{lemma QSC-RSI}, $F$ is $\left( \gamma, \frac{\mu a}{\gamma} - \frac{L}{2} \right)$-strongly quasar convex as long as $\gamma > \frac{2 \mu a}{L}$.
% By Lemma \ref{lemma RSI-(QG+QC)} it is $\mu^2 a$-quadratic growth and $\frac{2\mu^2}{L}a$-quasar convex. Finally by Lemma \ref{lemma (qg+qc)-SQC} it is $(\frac{\mu^2}{L}a,\mu^2 a)$-strongly quasar convex. \\
 Obviously, if there exists $x \neq x^\ast$ such that (\ref{UAAC CONDITION}) does not hold, then RSI can not hold. As strong quasar convexity implies RSI (Lemma \ref{lemma QSC-RSI}), this is a necessary condition for theorem \ref{theorem PL} to hold.
\end{proof}
Finally, we show that strong quasar convexity implies the PL condition.
\begin{proposition}
    Let $F$ be $(\gamma,\mu)$-strongly quasar convex. It is then $\mu \gamma^2$-PL.
\end{proposition}
\begin{proof}

We have 
\begin{align}
    \frac{1}{2}\lVert \frac{1}{\gamma \sqrt{\mu}}\nabla F(x) + \sqrt{\mu}(x^\ast - x) \rVert^2 = \frac{1}{2\gamma^2 \mu}\lVert \nabla F(x) \rVert^2 + \frac{1}{\gamma}\langle \nabla F(x),x^\ast - x \rangle + \frac{\mu}{2} \lVert x-x^\ast \rVert^2
\end{align}
Writting the definition of $(\gamma,\mu)$-strong quasar convexity, we have
\begin{align}
F^\ast &\geqslant F(x) + \frac{1}{\gamma} \langle \nabla F(x),x^\ast - x \rangle + \frac{\mu}{2} \lVert x-x^\ast \rVert^2\\
&= F(x) + \frac{1}{2}\lVert \frac{1}{\gamma \sqrt{\mu}}\nabla F(x) + \sqrt{\mu}(x^\ast - x) \rVert^2 -\frac{1}{2\gamma^2 \mu}\lVert \nabla F(x) \rVert^2\\
    \Rightarrow &\frac{1}{2\gamma^2 \mu}\lVert \nabla F(x) \rVert^2 \geqslant F(x)-F^\ast
\end{align}
\end{proof}
One can find a summary of the previous discussion in figure \ref{fig:graphs}.
\begin{figure}[!ht]
\centering
\resizebox{0.5\textwidth}{!}{%
\begin{circuitikz}
\tikzstyle{every node}=[font=\LARGE]
\node [font=\LARGE] at (14.75,16.5) {(SQC)};
\node [font=\LARGE] at (8,8.25) {(PL)};
\node [font=\LARGE] at (14.75,12.5) {(RSI)};
\node [font=\LARGE] at (14.75,8.25) {(EB)};
\draw [line width=1pt, ->, >=Stealth] (13.5,15.75) -- (8.75,9.25);
\draw [line width=1pt, ->, >=Stealth] (14.75,15.75) -- (14.75,13.25);
\draw [line width=1pt, ->, >=Stealth] (14.75,11.5) -- (14.75,9);
\draw [line width=1pt, ->, >=Stealth] (9,8.25) -- (13.75,8.25);
\draw [ color={rgb,255:red,138; green,226; blue,52}, line width=1pt, ->, >=Stealth, dashed] (15.5,9) .. controls (16,10.25) and (16,10.25) .. (15.5,11.5)node[pos=0.5,right, fill=white]{if (UAAC)};
\draw [ color={rgb,255:red,239; green,41; blue,41}, line width=1pt, ->, >=Stealth, dashed] (15.5,13.25) .. controls (16,14.5) and (16,14.5) .. (15.5,15.75)node[pos=0.5,right, fill=white]{if L-smooth};
\draw [ color={rgb,255:red,239; green,41; blue,41}, line width=1pt, ->, >=Stealth, dashed] (13.75,7.75) .. controls (11.5,7.25) and (11.25,7.25) .. (9,7.75)node[pos=0.5,below, fill=white]{if L-smooth};
\end{circuitikz}
}%
\caption{Summary of the Lemmas of Appendix \ref{appendix PL}. See Definition \ref{definition (strongly) quasar convex} for SQC (strongly quasar convex), Definition \ref{defRSI} for RSI, Definition \ref{EB} for EB (error bound), and Definition \ref{PL} for PL (Polyak-{\L}ojasiewcz). Solid lines are implications that hold without the need of adding another assumption. Red dashed lines are implications that hold under $L$-smooth assumption, while the green dashed line is for implication holding under the (\ref{UAAC CONDITION}) condition.}
\label{fig:graphs}
\end{figure}
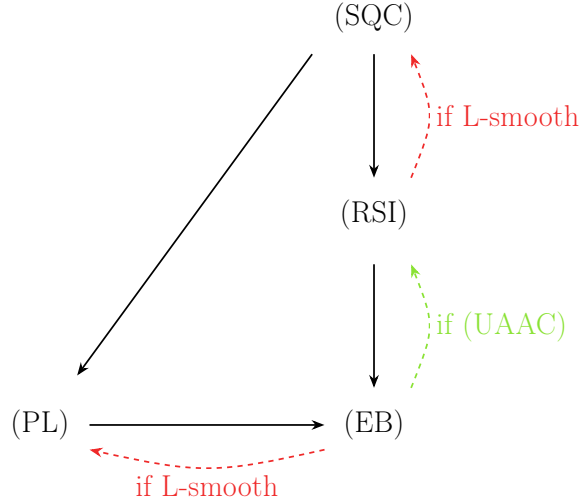
\section{Proofs of section 3}\label{appendix section 3}
\subsection{Differentiable strongly quasar convex}\label{appendix diff sqc}

\subsubsection{Proof of Theorem \ref{theorem 1}}\label{appendix proof theorem 1}\label{appendix theorem 1}
In this section, we detail the proof of Theorem~\ref{theorem 1} whose statement is recalled here: let $F$ be a $(\gamma,\mu)$-strongly quasar convex function for some $(\gamma,\mu)\in (0,1]\times \R_+^*$ having a $(\rho,L)$ curvature for some $L>0$ and $\rho\leqslant L$. Let $(x_n)_{n \in \mathbb{N}}$ generated by Algorithm \ref{algo}:
\begin{eqnarray*}
y_n &=& \alpha_nx_n + (1- \alpha_n)z_n\\
x_{n+1} &=& y_n - s\nabla F(y_n)\\
z_{n+1} &=& \beta_nz_n + (1 - \beta_n)y_n - \eta_n \nabla F(y_n)
\end{eqnarray*}
with parameters 
\begin{equation*}
    s\leqslant \frac{1}{L}, \quad\alpha_n = \frac{1}{1 + \sqrt{\mu s}}:=\alpha,\quad \beta_n = 1 - \gamma \sqrt{\mu s}:=\beta,\quad\eta_n = \frac{\sqrt{s}}{\sqrt{\mu }}:=\eta.
\end{equation*}
Assuming that $\rho \geqslant -\gamma \sqrt{\frac{\mu}{s}}$ we want to prove that:
        \begin{equation}
 \forall n\in \N, ~   F(x_{n}) - F^\ast \leqslant \frac{2}{\gamma} \left( 1 - \gamma \sqrt{\mu s} \right)^n(F(x_0)-F^\ast)\label{th1:result}
    \end{equation}
where $F^*=\min~F$. Let $x^*$ be the unique minimizer of $F$. Let be the following Lyapunov energy:
\begin{equation}
    E_n = F(x_n) - F^\ast + \frac{\mu}{2}\lVert z_n - x^\ast \rVert^2
\end{equation}
The main idea of the proof consists in finding parameters such that the following inequality holds
\begin{equation}
    E_{n+1} - E_n \leqslant cE_n
\end{equation}
with $c<0$ being as small as possible. We will then deduce the convergence rate \eqref{th1:result} by induction.

\paragraph{Step 1.} Since:
 \begin{align}
    E_{n+1} - E_n &= F(x_{n+1}) - F(x_n) + \frac{\mu}{2}\lVert z_{n+1} - x^\ast \rVert^2 - \frac{\mu}{2}\lVert z_n - x^\ast \rVert^2,
\end{align}
let us start by considering the right term:
\begin{eqnarray*}
\Delta_n &=& \lVert z_{n+1} - x^\ast \rVert^2 - \lVert z_n - x^\ast \rVert^2\\
&=& \lVert \beta z_n + (1 - \beta)y_n - \eta \nabla F(y_n) - x^\ast\rVert^2- \lVert z_n - x^\ast \rVert^2\\
&=& (\beta^2-1)\|z_n-x^*\|^2 +(1-\beta)^2\|y_n-x^*\|^2 +\eta^2\|\nabla F(y_n)\|^2 \\
&& +2 \beta  \langle z_n - x^\ast,(1-\beta)(y_n-x^\ast) - \eta \nabla F(y_n) \rangle - 2(1-\beta)\eta \langle \nabla F(y_n),y_n-x^*\rangle
\end{eqnarray*}
by construction of Algorithm~\ref{algo}. The tricky part here is to control the first scalar product: using the definition of Algorithm~\ref{algo}, we can rewrite $z_n = y_n + \frac{\alpha}{1-\alpha}(y_n -x_n)$, and thus
\begin{eqnarray*}
    &&\langle z_n - x^\ast,(1-\beta)(y_n-x^\ast) - \eta \nabla F(y_n) \rangle\\
   & &= \langle y_n - x^\ast,(1-\beta)(y_n-x^\ast) - \eta \nabla F(y_n) \rangle + \frac{\alpha}{1-\alpha}\langle y_n -x_n,(1-\beta)(y_n-x^\ast) - \eta \nabla F(y_n) \rangle\\
   & &= (1-\beta)\lVert y_n - x^\ast \rVert^2 - \eta\langle y_n -x^\ast ,\nabla F(y_n) \rangle - \frac{\alpha}{1-\alpha}\eta\langle y_n -x_n,\nabla F(y_n) \rangle\\
   & &+ \frac{\alpha}{1- \alpha}(1-\beta) \langle y_n -x_n, y_n -x^\ast \rangle
\end{eqnarray*}
Applying the relation $2\langle a,b\rangle = \|a+b\|^2-\|a\|^2-\|b\|^2$ to $a =y_n-x^*$ and $b=\frac{\alpha}{1-\alpha}(y_n-x_n)$, we get:
\begin{align}\label{CV GAP eq 2}
 \frac{\alpha}{1-\alpha}\langle y_n -x_n,y_n-x^\ast \rangle = \frac{1}{2}\lVert z_n - x^\ast \rVert^2 -  \frac{1}{2}\left(\frac{\alpha}{1-\alpha}\right)^2 \lVert y_n -x_n \rVert^2 -  \frac{1}{2}\lVert y_n -x^\ast \rVert^2,
\end{align}
so that:
\begin{eqnarray*}
    &&\langle z_n - x^\ast,(1-\beta)(y_n-x^\ast) - \eta \nabla F(y_n) \rangle\\
   & &= \frac{1-\beta}{2}\left(\lVert z_n - x^\ast \rVert^2 + \lVert y_n - x^\ast \rVert^2 - \left(\frac{\alpha}{1-\alpha}\right)^2 \lVert y_n -x_n \rVert^2 \right)- \eta\langle y_n -x^\ast ,\nabla F(y_n) \rangle\\
   & &- \frac{\alpha}{1-\alpha}\eta\langle y_n -x_n,\nabla F(y_n) \rangle
\end{eqnarray*}
and 
\begin{eqnarray*}
\Delta_n &=& -(1-\beta)\|z_n-x^*\|^2 +(1-\beta)\|y_n-x^*\|^2 +\eta^2\|\nabla F(y_n)\|^2 \\
&& -\beta(1-\beta)\left(\frac{\alpha}{1-\alpha}\right)^2\|y_n-x_n\|^2 -2\frac{\alpha\beta\eta}{1-\alpha}\langle \nabla F(y_n),y_n -x_n \rangle- 2\eta \langle \nabla F(y_n),y_n-x^*\rangle.
\end{eqnarray*}
Reinjecting $\Delta_n$ in the expression of $E_{n+1}-E_n$ and by definition of the Lyapunov energy $E_n$, we get:
\begin{align}
E_{n+1} - E_n &= -(1-\beta) E_n + F(x_{n+1}) - F^* -\beta\left(F(x_n)-F^*\right) + \frac{\mu}{2}(1-\beta)\|y_n-x^*\|^2 +\frac{\mu}{2}\eta^2\|\nabla F(y_n)\|^2\nonumber\\
& -\frac{\mu}{2}\beta(1-\beta)\left(\frac{\alpha}{1-\alpha}\right)^2\|y_n-x_n\|^2 -\frac{\alpha\beta\eta\mu}{1-\alpha}\langle \nabla F(y_n),y_n -x_n \rangle- \mu\eta \langle \nabla F(y_n),y_n-x^*\rangle.\label{proof:step1}
\end{align}

\paragraph{Step 2.} Let us now prove that for any $n\in \N$, we have: $E_{n+1}-E_n \leqslant -(1-\beta) E_n$ for some well-chosen values of the parameters $\beta$, $\eta$ and $\alpha$. 

Remember that $F$ is assumed strongly quasar convex, hence:
\begin{equation} \label{derivation difference proof 2}
\forall n\in \N,~\langle \nabla F(y_n),y_n-x^*\rangle \geqslant \gamma(F(y_n)-F^*) +\frac{\gamma\mu}{2}\|y_n-x^*\|^2
\end{equation}
and $L$-smooth which induces that:
\begin{equation} \label{derivation difference proof 3}
\forall s\leqslant \frac{1}{L},~\forall n\in \N,~\frac{s}{2}\|\nabla F(y_n)\|^2 \leqslant F(y_n)-F(x_{n+1})
\end{equation}
Reinjecting these two inequalities into $E_{n+1}-E_n$, we then get:
\begin{align}
     E_{n+1} - E_n &\leqslant -(1-\beta) E_n + F(x_{n+1}) - F^* -\beta\left(F(x_n)-F^*\right) + \frac{\mu}{s}\eta^2 ( F(y_n) - F(x_{n+1}))\nonumber\\%-\gamma \mu \eta(F(y_n)-F^\ast) \\
     & - \frac{\mu \beta (1-\beta)}{2}\left(\frac{\alpha}{1-\alpha}\right)^2 \lVert y_n -x_n \rVert^2 -  \frac{\alpha\beta\eta\mu}{1-\alpha} \langle \nabla F(y_n), y_n - x_n  \rangle  -\eta\mu\gamma (F(y_n)-F^*)\nonumber \\
     & +\frac{\mu}{2}(1-\beta - \gamma \eta \mu) \lVert y_n -x^\ast \rVert^2 \nonumber\\
     &\leqslant -(1-\beta)E_n + (  \frac{\mu}{s}\eta^2-\gamma \mu \eta)(F(y_n)-F^\ast) +(1 - \frac{\mu}{s}\eta^2) (F(x_{n+1})-F^\ast)\nonumber  \\
     &- \beta(F(x_n) - F^\ast) -  \frac{\alpha\beta\eta\mu}{1-\alpha} \langle \nabla F(y_n),y_n - x_n \rangle - \frac{\mu \beta (1-\beta)}{2}\left(\frac{\alpha}{1-\alpha}\right)^2 \lVert y_n -x_n \rVert^2 \label{derivation difference proof 1} \\
     &+\frac{\mu}{2}(1-\beta - \gamma \eta \mu) \lVert y_n -x^\ast \rVert^2\nonumber
\end{align}
Choosing now $\eta=\sqrt\frac{s}{\mu}$ and $\beta=1 -\gamma\eta\mu=1-\gamma\sqrt{\mu s}$ to cancel out the terms in $F(x_{n+1})-F^*$  and $\|y_n-x^*\|^2$, we deduce:
       \begin{align}
     E_{n+1} - E_n &\leqslant  -\gamma\sqrt{\mu s}E_n + (1-\gamma\sqrt{\mu s})\left(F(y_n) - F(x_n) \right)\nonumber \\
     & + (1-\gamma \sqrt{\mu s})\sqrt{\mu s} \frac{\alpha}{1-\alpha}\left(\langle \nabla F(y_n),x_n - y_n \rangle  - \frac{\alpha}{1-\alpha}\frac{\gamma \mu}{2} \lVert y_n -x_n \rVert^2 \right).\label{lastineq}
\end{align}
Suppose additionally that the lower curvature is bounded from below by $\rho$, hence: 
$$\forall n\in \N, ~F(y_n) + \langle \nabla F(y_n),x_n-y_n \rangle + \frac{\rho}{2}\lVert x_n - y_n \rVert^2 \leqslant F(x_n),$$
or equivalently:
$$\forall n\in \N, ~\langle \nabla F(y_n),x_n-y_n \rangle \leqslant F(x_n) - F(y_n)  - \frac{\rho}{2}\lVert x_n - y_n \rVert^2.$$
Injecting the very last inequality into \eqref{lastineq} we get:
\begin{align*}
         E_{n+1} - E_n &\leqslant  -\gamma\sqrt{\mu s}E_n + (1-\gamma\sqrt{\mu s})\left(1-\sqrt{\mu s} \frac{\alpha}{1-\alpha}\right)\left(F(y_n) - F(x_n) \right) \\
         & +  (1-\gamma \sqrt{\mu s})\frac{\sqrt{\mu s}}{2}  \frac{\alpha}{1-\alpha}\left(-\rho  - \frac{\alpha}{1-\alpha}\gamma \mu \right)  \lVert y_n -x_n \rVert^2 
\end{align*}
Lastly, choose $\alpha = \frac{1}{1+\sqrt{\mu s}}$ to cancel out the term in $F(y_n)-F(x_n)$, so that we finally get: 
   \begin{align*}
      E_{n+1} - E_n &=- \gamma\sqrt{\mu s}E_n + \frac{1-\gamma \sqrt{\mu s}}{2}\left(-\rho  - \gamma\frac{\sqrt{\mu}}{\sqrt{s}} \right)  \lVert y_n -x_n \rVert^2 \leqslant - \gamma\sqrt{\mu s}E_n
\end{align*}
provided that the lower curvature satisfies: $\rho \geqslant -\gamma\frac{\sqrt{\mu}}{\sqrt{s}}$.

\paragraph{Step 3.} To conclude, we proved that:
\begin{equation}
    E_{n+1} - E_n \leqslant -\gamma\sqrt{\mu s}E_n \Rightarrow E_{n+1} \leqslant (1- \gamma\sqrt{\mu s})E_n 
\end{equation}
By induction:
\begin{equation}
 \forall n\in \N,~   E_{n+1} \leqslant \left(1- \gamma\sqrt{\mu s} \right)^{n+1}E_0 =  \left(1- \gamma\sqrt{\mu s} \right)^{n+1}\left( F(x_0) - F^\ast + \frac{\mu}{2}\lVert x_0 - x^\ast \rVert^2 \right)
\end{equation}
In the last equality we use that by definition of Algorithm \ref{algo} $x_0 = z_0$.
We then use the fact that $(\gamma,\mu)$ strong convexity implies $\frac{\gamma \mu}{2-\gamma}$-quadratic growth (corollary 1 \cite{hinder2023nearoptimal}), \textit{i.e.}
\begin{equation}
 \forall x\in\R^d,~   F(x) - F^\ast \geqslant \frac{\gamma \mu }{2(2-\gamma)}\lVert x - x^\ast \rVert^2
\end{equation}
which finally yields to:
        \begin{equation}
        F(x_{n}) - F^\ast \leqslant \frac{2}{\gamma} \left( 1 - \gamma \sqrt{\mu s} \right)^n(F(x_0)-F^\ast).
    \end{equation}

\subsubsection{2 points scheme version of algorithm 1}\label{appendix 2 points}
Here we build upon the work \cite{lee2021geometric}, where they show there exists an equivalence between a 3 points and 2 points scheme version of Nesterov Accelerated gradient. We want to deduce a 2 points scheme from our 3 points one, but we can not directly apply their result because the 3 points algorithm they consider is slightly different.
\begin{proposition}
    The algorithm \begin{equation}\label{algo 3 points}
    \left\{
    \begin{array}{ll}
    y_n = \alpha_nx_n + (1- \alpha_n)z_n\\
    x_{n+1} = y_n - s\nabla F(y_n)\\
    z_{n+1} = \beta_nz_n + (1 - \beta_n)y_n - \eta_n \nabla F(y_n)
    \end{array}
\right.
\end{equation} can be written as the following 2 points scheme
\begin{equation}
    \left\{
    \begin{array}{ll}
        y_n = x_n + \frac{1 - \alpha_{n}}{1-\alpha_{n-1}}\alpha_{n-1} \beta_{n-1}(x_n - x_{n-1}) + (1-\alpha_{n})\left( \frac{\eta_{n-1}}{s} - \frac{\alpha_{n-1}\beta_{n-1}}{1-\alpha_{n-1}} -1 \right)(x_n - y_{n-1}) \\
        x_{n+1} = y_n - s\nabla F(y_n)
    \end{array}
\right.
\end{equation}
\end{proposition}
\begin{proof}
    We adapt the proof of Lemma 2 from \cite{lee2021geometric}, which is mainly based on Thales theorem. In their result, the $\alpha_n$ is $\frac{\varphi_n}{\varphi_{n-1}}$ and $\beta_nz_n + (1 - \beta_n)y_n$ is $z_n$.  The main idea is that if a vector $u$ is colinear to $v$, \textit{i.e.} $\exists \lambda \in \mathbb{R}$ $u = \lambda v$, then $\lambda = \frac{\lVert u \rVert}{\lVert v \rVert}$, which we can find using Thales theorem. As in the original proof, let us rewrite (\ref{algo 3 points}) as:
\begin{equation}\label{algo 3 point proof}
    \left\{
    \begin{array}{ll}
    x_n = \alpha_nx_{n-1}^+ + (1- \alpha_n)z_n\\
    z_{n+1} = \beta_nz_n + (1 - \beta_n)x_n - \eta_n \nabla F(x_n)
    \end{array}
\right.
\end{equation}
where $x_{n-1}^+ = x_{n-1} - s \nabla F(x_{n-1})$. We suppose $\nabla F(x_{n-1}) \neq 0$ (non degenerate case). We set $v_n = \beta_n z_n + (1-\beta_n)x_n$, and let's set A on the $[x_{n-1}^+x_n]$ segment such that $Ax_{n+1} \parallel x x_n^+ $. Let $B$ on $x_{n+1}^+x_n^+ \cap v_n z_{n+1}$. Since $A x_{n+1} \parallel B z_{n+1}$, we have by Thales theorem:
\begin{equation}
    \frac{\lVert B-A \rVert}{\lVert B-x_n^+ \rVert} = \frac{\lVert z_{n+1} -x_{n+1} \rVert}{\lVert z_{n+1} - x_n^+ \rVert}:= (\star)
\end{equation}
Then, by definition: 
\begin{align}
    z_{n+1} - x_{n+1} = \alpha_{n+1}(z_{n+1} - x_n^+) \Rightarrow \alpha_{n+1} = (\star)
\end{align}
As $B-A = \lambda (B-x_n^+)$ for some $\lambda \in \mathbb{R}$ (colinearity), we have $\lambda = \alpha_{n+1}$ and then 
\begin{equation}\label{3 point vers 2 point relation 1}
    B-A = \alpha_{n+1}(B-x_n^+)
\end{equation}Also, colinearity of $Bx_n^+$ and $x_n^+x_{n-1}^+$ together with $x_n x_n^+ \parallel v_n B$ leads to $B-x_n^+ = \lambda(x_n^+ - x_{n-1}^+)$, with
\begin{align}
 \lambda = \frac{\lVert B-x_n^+\rVert}{\lVert x_n^+ - x_{n-1}^+\rVert} = \frac{\lVert v_n - x_n \rVert}{\lVert x_n - x_{n-1}^+ \rVert}
\end{align}
where we have
\begin{align}
    &v_n - x_n = \beta_n z_n + (1-\beta_n)x_n - x_n = \beta_n(z_n - x_n) = \beta_n\alpha_n (z_n - x_{n-1}^+)\\
    &x_n - x_{n-1}^+ = (1 - \alpha_n)(z_n - x_{n-1}^+)
\end{align}
such that 
\begin{equation}\label{3 point vers 2 point relation 2}
    B-x_n^+ = \beta_n\frac{\alpha_n}{1 - \alpha_n}(x_n^+ - x_{n-1}^+)
\end{equation}
Combining (\ref{3 point vers 2 point relation 1}) and (\ref{3 point vers 2 point relation 2}), we get 
\begin{align}\label{A-xk+}
    A - x_n^+ &= (B - x_n^+) - (B-A) = (B-x_n^+) = (B-x_n^+) - \alpha_{n+1}(B-x_n^+)\\
    &= (1-\alpha_{n+1})(B-x_n^+) = \beta_n \frac{(1-\alpha_{n+1})\alpha_n}{1-\alpha_n}(x_n^+ - x_{n-1}^+)
\end{align} 
Then, we study $x_{n+1} - A$. As $Ax_{n+1} \parallel B z_{n+1}$ we have by Thales theorem
\begin{align}
    \frac{\lVert x_{n+1} - A\rVert}{\lVert z_{n+1} - B\rVert} = \frac{\lVert x_{n+1}-x_n^+\rVert}{\lVert z_{n+1} - x_{n}^+\rVert} = 1 - \alpha_{n+1}
\end{align}
Last equality because $x_{n+1} - x_n^+ = (1 - \alpha_{n+1})(z_{n+1} -x_n^+)$. We then have \begin{equation}\label{relation 3 et 4 3 points vers 2 points}
    x_{n+1} - A = (1 - \alpha_{n+1})(z_{n+1}- B) = (1 - \alpha_{n+1})((z_{n+1}-v_n) - (B-v_n))
\end{equation}
Where we recall $v_n = \beta_n z_n + (1-\beta_n)x_n$. We have: 
\begin{equation}\label{3 points vers 2 points relation 3}
z_{n+1} - v_n = (x_n^+ - x_n)\frac{\eta_n}{s}
\end{equation}
Then we use $x_n x_n^+ \parallel v_n B$ to us Thales theorem once again: 
\begin{align}
    \frac{\lVert B-v_n\rVert}{\lVert x_n^+ - x_n\rVert} = \frac{\lVert v_n - x_{n-1}^+\rVert}{\lVert x_n -  x_{n-1}^+\rVert}:= (\star \star)
\end{align}
We have $v_n - x_{n-1}^+ = \beta_n z_n + (1-\beta_n)x_n - x_{n-1}^+$. We have also
\begin{align}
    &x_n = \alpha_n x_{n-1}^+ + (1-\alpha_n)z_n \\
    \Rightarrow &\beta_nz_n - \beta_n x_n = \alpha_n \beta_n (z_n - x_{n-1}^+)\\
    \Rightarrow & \beta_n z_n + (1-\beta_n)x_n - x_{n-1}^+ = \alpha_n \beta_n (z_n - x_{n-1}^+) + x_n - x_{n-1}^+\\
    \Rightarrow &V_n - x_{n-1}^+ = (\frac{\alpha_n \beta_n}{1-\alpha_n}+1)(x_n - x_{n-1}^+)
\end{align}
This induces that $(\star \star) = (\frac{\alpha_n \beta_n}{1-\alpha_n}+1)$ and then
\begin{equation}\label{3 points vers 2 points relation 4}
    B-v_n = \left(\frac{\alpha_n \beta_n}{1-\alpha_n}+1\right)(x_n^+-x_n)
\end{equation}
Finally, injecting (\ref{3 points vers 2 points relation 3}) and (\ref{3 points vers 2 points relation 4}) in (\ref{relation 3 et 4 3 points vers 2 points}), we get 
\begin{align}\label{xk+1 - A}
    x_{n+1} - A &= (1 - \alpha_{n+1})\left(\frac{\eta_n}{s}(x_n^+ - x_n)- (\frac{\alpha_n \beta_n}{1-\alpha_n}+1)(x_n^+-x_n)\right)\\
    &= (1 - \alpha_{n+1})\left(\frac{\eta_n}{s} - \frac{\alpha_n \beta_n}{1-\alpha_n}+1\right)(x_n^+-x_n)
\end{align}
We can conclude by combining (\ref{A-xk+}) and (\ref{xk+1 - A}) that 
\begin{equation}
    x_{n+1} = x_n^+ +\beta_n \frac{(1-\alpha_{n+1})\alpha_n}{1-\alpha_n}(x_n^+ - x_{n-1}^+) + (1 - \alpha_{n+1})\left(\frac{\eta_n}{s} - \frac{\alpha_n \beta_n}{1-\alpha_n}+1\right)(x_n^+-x_n)
\end{equation}
\end{proof}

\begin{corollaire}
The algorithm 1 with parameters $s\leqslant \frac{1}{L}$, $\alpha_n = \frac{1}{1 + \sqrt{\mu s}}$, $\beta_n = 1 - \gamma \sqrt{\mu s}$ and $\eta_n = \frac{\sqrt{s}}{\sqrt{\mu }}$ can be written as the following 2 points scheme
    \begin{equation}
    \left\{
    \begin{array}{ll}
        y_n = x_n + \frac{1 - \gamma \sqrt{\mu s}}{1 + \sqrt{\mu s}} (x_n - x_{n-1}) + \frac{\sqrt{\mu s}}{1 + \sqrt{\mu s}}(\gamma -1 )(x_n - y_{n-1}) \\
        x_{n+1} = y_n - s\nabla F(y_n)
    \end{array}
\right.
\end{equation}
\end{corollaire}
\begin{proof}
    We just apply previous result with $\alpha_n = \frac{1}{1 + \sqrt{\mu s}}$, $\beta_n = 1 - \gamma \sqrt{\mu s}$ et $\eta_n = \frac{\sqrt{s}}{\sqrt{\mu }}$.
\end{proof}
\subsubsection{Proof of theorem \ref{theorem 2}}\label{appendix non diff theorem 2}
In this section, we detail the proof of Theorem~\ref{theorem 2} whose statement is: let $F = f+g$ where $f: \mathbb{R}^d \to \mathbb{R}$ is a $L$-smooth function for some $L>0$ and $g: \mathbb{R}^d \to \mathbb{R}$ is convex, proper, lower semi-continuous. Assume that $F$ has a non empty set of minimizers and that $f$ is $(1,\mu)$-strongly quasar convex with respect to $x^\ast_F \in \arg \min~F$  and $(\rho,L)$-curvatured for some $\rho\leqslant L$. Let $(x_n)_{n\in \N}$ be generated by Algorithm~\ref{algo non diff}:
\begin{eqnarray*}
y_n &=& \alpha_nx_n + (1- \alpha_n)z_n\\
x_{n+1} &=& \text{prox}_{sg}(y_n - s\nabla f(y_n)):= T_s(y_n)\\
z_{n+1} &=& \beta_nz_n + (1 - \beta_n)y_n - \frac{\eta_n}{s}(y_n - T_s(y_n))
\end{eqnarray*}
with parameters 
\begin{equation*}
    s\leqslant \frac{1}{L}, \quad\alpha_n = \frac{1}{1 + \sqrt{\mu s}}:=\alpha,\quad \beta_n = 1 - \gamma \sqrt{\mu s}:=\beta,\quad\eta_n = \frac{\sqrt{s}}{\sqrt{\mu }}:=\eta.
\end{equation*}
Assuming that $\rho \geqslant - \sqrt{\frac{\mu}{s}}$ we want to prove that:
        \begin{equation}
 \forall n\in \N, ~   F(x_{n}) - F^\ast \leqslant 2\left( 1 - \sqrt{\mu s} \right)^n(F(x_0)-F^\ast)\label{th3:result}
    \end{equation}
where $F^*=\min~F$. Let $x^*$ be the unique minimizer of $F$. As for the proof of Theorem \ref{theorem 1}, we introduce the following Lyapunov energy:
\begin{equation}
    E_n = F(x_n) - F^\ast + \frac{\mu}{2}\lVert z_n - x^\ast \rVert^2
\end{equation}
and we seek the parameters et conditions for which the following inequality holds:
\begin{equation}
    E_{n+1} - E_n \leqslant cE_n
\end{equation}
with $c<0$ being as small as possible. We will then deduce the convergence rate \eqref{th3:result} by induction.
\begin{proof}
The proof is very similar to Theorem \ref{theorem 1}. Let us consider the same Lyapunov energy:
\begin{equation}
        E_{n} = F(x_n) - F^\ast + \frac{\mu}{2}\lVert z_n - x^\ast_F \rVert^2
\end{equation}
\paragraph{Step 1.} We start by calculating $E_{n+1}-E_n$, and the exact same computations as in Step 1 for Theorem \ref{theorem 1} leads to:
\begin{align}
E_{n+1} - E_n &= -(1-\beta) E_n + F(x_{n+1}) - F^* -\beta\left(F(x_n)-F^*\right) + \frac{\mu}{2}(1-\beta)\|y_n-x^*_F\|^2\nonumber \\
& +\frac{\mu}{2s^2}\eta^2\|y_n-T_s(y_n)\|^2 -\frac{\mu}{2}\beta(1-\beta)\left(\frac{\alpha}{1-\alpha}\right)^2\|y_n-x_n\|^2 \nonumber\\
&+ \frac{\mu\eta}{s} \langle y_n-T_s(y_n),x^*_F-y_n\rangle +\frac{\alpha\beta\eta\mu}{(1-\alpha)s}\langle y_n-T_s(y_n),x_n -y_n \rangle,
\end{align}
replacing $\nabla f(y_n)$ by the composite gradient $\frac{1}{s}(y_n - T_s(y_n))$, where $T_s(y_n):=\text{prox}_{sg}(y_n - s \nabla f(y_n))$.

\paragraph{Step 2.} Let us now prove that for any $n\in \N$, we have: $E_{n+1} -E_n \leqslant -(1-\beta)E_n$ for some well-chosen values of the parameters $\beta$, $\eta$ and $\alpha$.

    To control the scalar products, first note that:      
    \begin{align}\label{I-i}
        2\langle y_n-T_s(y_n),x^*_F-y_n\rangle = \lVert T_s(y_n) - x^\ast_F \rVert^2 - \lVert y_n - T_s(y_n) \rVert^2 - \lVert y_n - x^\ast_F \rVert^2.
    \end{align}
    Combining the prox-grad inequality (\cite[Theorem 10.16]{beck2017}): for all $(x,y)\in \R^d\times\R^d$,
    \begin{equation}\label{prox-grad}\tag{Prox-Grad}
     F(x) - F(T_s(y)) \geqslant \frac{1}{2s} \lVert x-T_s(y) \rVert^2 - \frac{1}{2s}\lVert x-y \rVert^2 + f(x) - f(y) - \langle \nabla f(y),x-y \rangle,
    \end{equation}
 applied at $x = x^\ast_F$ and $y = y_n$, and the definition of strong quasar convexity in the sense of Definition~\ref{def:SQC:nonminimizer}:
\begin{align*}\label{I-iii}
\forall n\in\N,~f(x^\ast_F) - f(y_n) - \langle \nabla f(y_n),x^\ast_F -y_n \rangle \geqslant \frac{\mu}{2}\lVert x^\ast_F - y_n \rVert^2,
\end{align*}
we prove that for any $n\in N$,
\begin{eqnarray*}
2s\left(F^* - F(T_s(y_n))\right) &\geqslant &\lVert x^\ast_F-T_s(y_n) \rVert^2 - \lVert x^\ast_F-y_n \rVert^2 + 2s\left(f(x^\ast_F) - f(y_n) -\langle \nabla f(y_n),x^\ast_F -y_n \rangle\right)\\
&\geqslant & \lVert x^\ast_F-T_s(y_n) \rVert^2 - \lVert x^\ast_F-y_n \rVert^2 + \mu s\lVert x^\ast_F - y_n \rVert^2.
\end{eqnarray*}
Hence, reinjecting into \eqref{I-i} and remembering $x_{n+1}=T_s(y_n)$, we get:
\begin{equation}
\forall n\in \N,~\langle y_n - T_s(y_n),x^\ast_F - y_n\rangle  \leqslant -s\left( F(x_{n+1})-F^\ast \right) - \frac{\mu s}{2}\lVert y_n -x^\ast_F \rVert^2 - \frac{1}{2}\lVert y_n - x_{n+1} \rVert^2.\label{controlI}
\end{equation}
%    \begin{equation}\label{I-ii}
%        F^\ast - F(T_s(y_n)) \geqslant \frac{1}{2s} \lVert x^\ast_F-T_s(y_n) \rVert^2 - \frac{1}{2s}\lVert x^\ast_F-y_n \rVert^2 + f(x^\ast_F) - f(y_n) -\langle \nabla f(y_n),x^\ast_F -y_n \rangle
%    \end{equation}
%    Here we see that we can not apply strong quasar convexity in the sense of Definition \ref{definition (strongly) quasar convex}. We recall then that we chose instead this assumption:
%    \begin{align}\label{I-iii}
%        f(x^\ast_F) - f(y_n) - \langle \nabla f(y_n),x^\ast_F -y_n \rangle \geqslant \frac{\mu}{2}\lVert x^\ast_F - y_n \rVert^2
%    \end{align}
%    Mixing lines (\ref{I-i}), (\ref{I-ii}) and (\ref{I-iii}) we get 
%    \begin{align}
%        \langle y_n - T_s(y_n),x^\ast_F - y_n\rangle = I \leqslant -s\left( F(T_s(y_n))-F^\ast \right) - \frac{\mu s}{2
%    }\lVert y_n -x^\ast_F \rVert^2 - \frac{1}{2}\lVert y_n - T_s(y_n) \rVert^2
%    \end{align}

Similarly, consider then the second scalar product: 
    %\begin{align}\label{II-i}
    %    2\langle x_n - y_n, \ y_n - T_s(y_n) \rangle = \lVert T_s(y_n) - x_n \rVert^2 - \lVert y_n - T_s(y_n) \rVert^2 - \lVert y_n - x_n\rVert^2
    %\end{align}
\begin{eqnarray}
2\langle x_n - y_n, \ y_n - T_s(y_n) \rangle &=& \lVert T_s(y_n) - x_n \rVert^2 - \lVert y_n - T_s(y_n) \rVert^2 - \lVert y_n - x_n\rVert^2\\
&\leqslant & 2s\left(F(x_n) -F(T_s(y_n))\right)+ 2s(f(y_n)-f(x_n) +\langle \nabla f(y_n),x_n-y_n\rangle) \nonumber\\
&&-\|y_n-T_s(y_n)\|^2\\
&\leqslant & 2s\left(F(x_n) -F(x_{n+1})\right)+ 2s(f(y_n)-f(x_n) +\langle \nabla f(y_n),x_n-y_n\rangle) \nonumber\\
&&-\|y_n-x_{n+1}\|^2\label{controlII}
\end{eqnarray}
using again \eqref{prox-grad} evaluated at $x = x_n$ and $y = y_n$ and $x_{n+1}=T_s(y_n)$.\\
%    \begin{align}
%        II \leqslant -s\left( F(T_s(y_n)) - F(x_n)\right) - \frac{1}{2}\lVert y_n - T_s(y_n) \rVert^2 + s\left(f(y_n) + \langle \nabla f(y_n),x_n - y_n \rangle - f(x_n) \right)
%    \end{align}

Reinjecting \eqref{controlI} and \eqref{controlII} into the expression of $E_{n+1}-E_n$ obtained at the end of Step 1, we get:
%\begin{align}
%    E_{n+1}-E_n &\leqslant F(x_{n+1})-F(x_n) + (1-\beta)\frac{\mu}{2}\lVert z_n - x^\ast_F \rVert^2 + \frac{\mu}{2}\left( (1-\beta)^2 + \beta(1-\beta) \right)\lVert y_n - x^\ast_F \rVert^2\\
%    &+\frac{\mu}{2s^2}\eta^2 \lVert y_n - T_s(y_n) \rVert^2 + \frac{\mu \eta}{s} \left(-s\left( F(T_s(y_n))-F^\ast \right) - \frac{\mu s}{2}\lVert y_n -x^\ast_F \rVert^2 - \frac{1}{2}\lVert y_n - T_s(y_n) \rVert^2 \right) \\
%    &+ \mu \beta \frac{\alpha}{1-\alpha}\frac{\eta}{s} \left( -s\left( F(T_s(y_n)) - F(x_n)\right) - \frac{1}{2}\lVert y_n - T_s(y_n) \rVert^2 \right. \\
%    &\left. +s\left(f(y_n) + \langle \nabla f(y_n),x_n - y_n \rangle - f(x_n) \right) \right)- \frac{\mu \beta (1-\beta)}{2}\left( \frac{\alpha}{1-\alpha} \right)^2\lVert x_n - y_n \rVert^2 \\
%    &= \left( \mu \beta \frac{\alpha}{1-\alpha}\eta-1 \right)(F(x_n) - F^\ast) + (1-\beta)\frac{\mu}{2}\lVert z_n - x^\ast_F \rVert^2 + \frac{\mu}{2}\left( (1-\beta - \eta \mu \right)\lVert y_n - x^\ast_F \rVert^2\\
%    &+\frac{\mu \eta}{2s} \left( \frac{\eta}{s} - 1 - \beta \frac{\alpha}{1-\alpha} \right)\lVert y_n - T_s(y_n)\rVert^2 + \left( 1 - \mu \eta-\mu \beta \frac{\alpha}{1-\alpha}\eta \right)(F(x_{n+1})-F^\ast)\\
%    &+\mu \beta \frac{\alpha}{1-\alpha}\eta\left(f(y_n) + \langle \nabla f(y_n),x_n - y_n \rangle - f(x_n) \right)- \frac{\mu \beta (1-\beta)}{2}\left( \frac{\alpha}{1-\alpha} \right)^2\lVert x_n - y_n \rVert^2
%\end{align}
\begin{eqnarray*}
E_{n+1}-E_n &\leqslant & -(1-\beta)E_n + \left(1-\mu\eta-\mu \beta \frac{\alpha}{1-\alpha}\eta\right)(F(x_{n+1})-F^*) + \beta\left(\frac{\alpha\eta\mu}{1-\alpha}-1\right)(F(x_n)-F^*)\\
&& +\frac{\mu}{2}\left(1-\beta-\mu\eta\right)\|y_n-x^*_F\|^2 + \frac{\mu\eta}{2s}\left(\frac{\eta}{s}-1 -\frac{\alpha\beta}{1-\alpha}\right)\|y_n-x_{n+1}\|^2\\
&&-\frac{\mu}{2}\beta(1-\beta)\left(\frac{\alpha}{1-\alpha}\right)^2\|y_n-x_n\|^2 -\frac{\alpha\beta\eta\mu}{1-\alpha}\left(f(x_n)-f(y_n)-\langle\nabla f(y_n),y_n-x_n\rangle\right)
\end{eqnarray*}
As for Theorem~\ref{theorem 1}, choose: $\eta = \frac{\sqrt{s}}{\sqrt{\mu }}$, $\beta = 1 - \eta\mu=1- \sqrt{\mu s}$ and $\alpha = \frac{1}{1 + \eta\mu}=\frac{1}{1 + \sqrt{\mu s}}$ to cancel out the terms in $F(x_{n+1})-F^*$, $\|y_n-x^*_F\|^2$, $F(x_n)-F^*$ and $\|y_n-x_{n+1}\|^2$. We then get:
\begin{align}
    E_{n+1}-E_n \leqslant -(1-\beta)E_n + \beta\left(f(y_n) + \langle \nabla f(y_n),x_n - y_n \rangle - f(x_n) \right)-\beta \frac{\sqrt{\mu}}{2\sqrt{s}}\lVert x_n - y_n \rVert^2.
\end{align}
Finally, assuming additionally that $f$ is $(\rho,L)$-curvatured for some $\rho\leqslant L$, observe that:
$$\forall n\in \N, f(y_n)+ \langle \nabla f(y_n),x_n - y_n \rangle - f(x_n) \leqslant -\frac{\rho}{2}\|x_n-y_n\|^2,$$
which induces:
%$F$ is $(-\sqrt{\frac{\mu}{s}})$-curvatured, inducing that
\begin{equation*}
\forall n\in \N,~    E_{n+1}-E_n \leqslant -(1-\beta)E_n -\frac{\beta}{2}\left(\rho +\frac{\sqrt{\mu}}{\sqrt{s}}\right) \|x_n-y_n\|^2\\
\end{equation*}
Provided that $\rho\geqslant -\sqrt\frac{\mu}{s}$, we finally obtain the expected inequality, namely: $E_{n+1}-E_n \leqslant -\sqrt{\mu s}E_n$ for all $n\in \N$, and we can conclude the proof exactly the same way as for Theorem \ref{theorem 1}, Step 3.
\end{proof}

\section{Continuous analysis through High resolution ODEs}\label{appendix high res}
\paragraph{Derivation of ODE}
Recall that the algorithm we prove convergence in Theorem \ref{theorem 1} can be written
    \begin{equation}
    \left\{
    \begin{array}{ll}
        y_n = x_n + \frac{1 - \gamma \sqrt{\mu s}}{1 + \sqrt{\mu s}} (x_n - x_{n-1}) + \frac{\sqrt{\mu s}}{1 + \sqrt{\mu s}}(\gamma -1 )(x_n - y_{n-1}) \\
        x_{n+1} = y_n - s\nabla F(y_n)
    \end{array}
\right.
\end{equation}
Writting only with respect to $y_n$, we get
\begin{equation}\label{eq haute res}
    y_{n+1} = y_{n} + \frac{1 - \gamma \sqrt{\mu s}}{1 + \sqrt{\mu s}}(y_{n}-y_{n-1}) - s\left(1 +\frac{\sqrt{\mu s}}{1 + \sqrt{\mu s}}(\gamma-1 )  \right)\nabla F(y_{n}) -s \frac{1 - \gamma \sqrt{\mu s}}{1 + \sqrt{\mu s}}(\nabla F(y_{n}) - \nabla F(y_{n-1}))
\end{equation}
The following development will be very close to the one introduced in \cite{shi2018understanding}. We assume there exists a smooth curve $X$ such that $X(t_n) = y_n$, where $t_n = n\sqrt{s}$. By Taylor development, we have 
\begin{align}
    &y_{n+1} = X(t_{n+1}) = X(t_n) + \sqrt{s}\dot{X}(t_n) + \frac{s}{2}\ddot{X}(t_n) + \frac{\sqrt{s}^3}{6}\dddot{X}(t_n) + \bigO(s^2)\\
    &    y_{n-1} = X(t_{n-1}) = X(t_n) - \sqrt{s}\dot{X}(t_n) + \frac{s}{2}\ddot{X}(t_n) - \frac{\sqrt{s}^3}{6}\dddot{X}(t_n) + \bigO(s^2)
\end{align}
Another Taylor development gives
\begin{equation}
    \nabla F(y_n) - \nabla F(y_{n-1}) = \nabla^2 F(X(t_n))\dot{X}(t_n)\sqrt{s} + \bigO(s)
\end{equation}
Multiplying both sides of (\ref{eq haute res}) by $\frac{1+ \sqrt{\mu s}}{1-\gamma\sqrt{\mu s}}\frac{1}{s}$, we get
\begin{align}
    \frac{y_{n+1} + y_{n-1} - 2y_n}{s} + \frac{(1+\gamma)\sqrt{\mu s}}{1-\gamma\sqrt{\mu s}}\frac{y_{n+1} - y_{n}}{s} + \nabla F(y_n) - \nabla F(y_{n-1}) + \frac{1+\gamma \sqrt{\mu s}}{1 - \gamma \sqrt{\mu s}}\nabla F(y_{n-1}) = 0
\end{align}
Using Taylor developments above, we have
\begin{align}
    \ddot{X}(t_n) &+ \bigO(s) + \frac{(1+\gamma)\sqrt{\mu }}{1-\gamma\sqrt{\mu s}}\left[ \dot{X}(t_n)+ \frac{1}{2}\ddot{X}(t_n)\sqrt{s} + \bigO(s) \right] 
    \\
    &+ \nabla^2 F(X(t_n))\dot{X}(t_n)\sqrt{s} + \bigO(s) + \left( \frac{1+\gamma\sqrt{\mu s}}{1-\gamma\sqrt{\mu s}} \right)\nabla F(X(t_n))=0
\end{align}
Multiplying both sides by $1-\gamma \sqrt{\mu s}$ and ignoring $\bigO(s)$ terms, we get that (\ref{eq haute res}) is a discretization of:
\begin{equation}\label{NAG-SQC-ODE1}\tag{NAG-SQC-ODE}
    (1+\frac{1-\gamma}{2}\sqrt{\mu s})\ddot{X}(t) + (1+\gamma)\sqrt{\mu}\dot{X}(t) + \sqrt{s}\nabla^2 F(X(t))\dot{X} + (1+\gamma\sqrt{\mu s})\nabla F(X(t)) = 0
\end{equation}
\begin{theorem*}
    Let $F$ be $C^2$, $(\gamma,\mu)-$strongly quasar convex and $L$-smooth. Assume X is solution of (\ref{NAG-SQC-ODE1}) with $0 \leqslant s \leqslant \frac{1}{L}$, $X(0) = X_0$ and $\dot{X}(0) =0 $. Then:
       \begin{equation}
            F(X(t))-F^\ast \leqslant  K_0(\gamma,\mu,L,s)\frac{1}{\gamma}(F(X_0)-F^\ast)  e^{-\gamma\frac{\sqrt{\mu}}{2}t}
        \end{equation}
        where $ K_0(\gamma,\mu,L,s) \leqslant 7$.
\end{theorem*}
\begin{proof}
    
We rewrite the ODE (\ref{NAG-SQC-ODE1}) the following way.
\begin{equation}
    \upsilon \ddot{X}(t) + (1+\gamma)\sqrt{\mu}\dot{X}(t) + \sqrt{s}\nabla^2 F(X(t))\dot{X}(t) + (1 + \gamma\sqrt{\mu s})\nabla F(X(t)) = 0
\end{equation}
Where $\upsilon = 1 + \frac{1-\gamma}{2}\sqrt{\mu s}$.
Set the following Lyapunov function: 
    \begin{equation}
        \mathcal{E}(t) = \delta(F(X(t))-F^\ast) + \frac{1}{2}\lVert \upsilon\dot{X}(t) + \lambda(X(t) - x^\ast) + \sqrt{s}\nabla F(X(t)) \rVert^2
    \end{equation}
    To lighten the following computations, we write $X(t)$ as $X$, and we do the same for the first and second derivatives of $X$.
    We have 
    \begin{align}
        \dot{\mathcal{E}}(t) = \delta \langle \dot{X},\nabla F(X) \rangle  + \langle \upsilon \dot{X}(t) + \lambda(X(t) - x^\ast) + \sqrt{s}\nabla F(X(t)), \upsilon \ddot{X} + \lambda\dot{X} + \sqrt{s}\nabla^2 F(X)\dot{X}\rangle.
    \end{align}
     Injecting (\ref{NAG-SQC-ODE1}), we get 
        \begin{align}
            \dot{\mathcal{E}}(t) &= \delta\langle \dot{X},\nabla F(X) \rangle + \langle \upsilon \dot{X}(t) + \lambda(X(t) - x^\ast) \\
            &+\sqrt{s}\nabla F(X(t)),(\lambda - (1+\gamma)\sqrt{\mu})\sqrt{\mu}\dot{X} - (1 + \gamma\sqrt{\mu s})\nabla F(X)\rangle\\
            &=\delta \langle \dot{X},\nabla F(X) \rangle + \upsilon(\lambda - (1+\gamma)\sqrt{\mu}) (\lVert \dot{X} \rVert^2 \\
            &- \upsilon(1 + \gamma\sqrt{\mu s})\langle \dot{X},\nabla F(X)\rangle + \lambda (\lambda - (1+\gamma)\sqrt{\mu}) \langle X-x^\ast,\dot{X}\rangle\\
            &- \lambda(1 + \gamma\sqrt{\mu s})\langle X - x^\ast,\nabla F(X) \rangle + (\lambda - (1+\gamma)\sqrt{\mu})\sqrt{s}\langle \nabla F, \dot{X}\rangle - \sqrt{s}(1 + \gamma \sqrt{\mu s}) \lVert \nabla F(X) \rVert^2
            %\\
            %&=-\gamma \sqrt{\mu}\left( \lVert \dot{X} \rVert^2 + \sqrt{\mu}\langle X - x^\ast, \dot{X}\rangle + \frac{1}{\gamma}(1+\sqrt{\mu s})\langle X - x^\ast,\nabla F(X) \rangle + s\lVert \nabla F(X) \rVert^2\right) - \sqrt{s}\lVert \nabla F(X) \rVert^2
        \end{align}
        We set $(\lambda - (1+\gamma)\sqrt{\mu}) = -\upsilon\gamma \sqrt{\mu} $. Then to cancel $\langle \dot{X},\nabla F(X) \rangle $ terms, we set 
        \begin{align}
            \delta = \left( \upsilon(1+\gamma\sqrt{\mu s}) +  \upsilon\gamma \sqrt{\mu s} \right) = \upsilon(1 + 2 \gamma \sqrt{\mu s})
        \end{align}
        Then we get
        \begin{align}
            \dot{\mathcal{E}}(t) &\leqslant -\gamma\sqrt{\mu}\left(\upsilon^2 \lVert \dot{X} \rVert^2 + \lambda \upsilon\langle X-x^\ast,\dot{X} \rangle + \frac{\lambda}{\gamma \sqrt{\mu}}(1+\gamma \sqrt{\mu s})\langle X-x^\ast,\nabla F(X) \rangle  + s \lVert \nabla F(X) \rVert^2\right) \\
            &- \sqrt{s}\lVert \nabla F(X) \rVert^2 
        \end{align}
        We use strong quasar convexity.
        \begin{align}
            \dot{\mathcal{E}}(t) &\leqslant -\gamma\sqrt{\mu}\left(\upsilon^2 \lVert \dot{X} \rVert^2 + \lambda \upsilon\langle X-x^\ast,\dot{X} \rangle + \frac{\lambda}{\sqrt{\mu}}(1+\gamma \sqrt{\mu s})(F(X)-F^\ast)\right.\\
            & \left. + \frac{\lambda \sqrt{\mu}}{2} (1+\gamma \sqrt{\mu s})\lVert X-x^\ast \rVert^2+ s \lVert \nabla F(X) \rVert^2\right) - \sqrt{s}\lVert \nabla F(X) \rVert^2\\
            &= -\gamma\sqrt{\mu}\left(\frac{1}{2}\upsilon^2 \lVert \dot{X} \rVert^2 + \lambda \upsilon\langle X-x^\ast,\dot{X} \rangle + \frac{\lambda^2}{2} \lVert X-x^\ast \rVert^2 \right)\\
             & \left.  + \frac{\upsilon}{2}(1+2\gamma\sqrt{\mu s})(F(X)-F^\ast)+ s \lVert \nabla F(X) \rVert^2- \sqrt{s}\lVert \nabla F(X) \rVert^2\right)\\
             &-\gamma \sqrt{\mu} \left( \frac{\lambda \sqrt{\mu}}{2} (1+\gamma \sqrt{\mu s}) - \frac{\lambda^2}{2}\right)\lVert X-x^\ast \rVert^2 \\
             &- \gamma \sqrt{\mu}\left(\frac{\lambda}{\sqrt{\mu}}(1+\gamma \sqrt{\mu s}) - \frac{\upsilon}{2}(1+2\gamma\sqrt{\mu s})\right)(F(X)-F^\ast)
             -\gamma\frac{\sqrt{\mu }\upsilon^2}{2}\lVert \dot{X}\rVert^2
        \end{align}
        We have $\frac{1}{2}\upsilon^2 \lVert \dot{X} \rVert^2 + \lambda \upsilon\langle X-x^\ast,\dot{X} \rangle + \frac{\lambda^2}{2} \lVert X-x^\ast \rVert^2 = \lVert \upsilon \dot{X} + \lambda(X - x^\ast) \rVert^2$.
           Then, we use:
        \begin{align}
            &\frac{1}{2}\lVert \upsilon \dot{X} + \lambda(X - x^\ast) + \sqrt{s}\nabla F(X) \rVert^2 \leqslant \lVert \upsilon \dot{X} + \lambda(X - x^\ast) \rVert^2 + s \lVert \nabla F(X) \rVert^2\\
            \Rightarrow-&\frac{1}{4}\lVert  \upsilon\dot{X} + \lambda(X - x^\ast) + \sqrt{s}\nabla F(X) \rVert^2 \geqslant -\frac{1}{2}\lVert \upsilon \dot{X} + \lambda(X - x^\ast) \rVert^2 -\frac{s}{2}\lVert \nabla F(X) \rVert^2
        \end{align}
        This leads to:
        \begin{align}
            \dot{\mathcal{E}}(t) &\leqslant -\gamma \frac{\sqrt{\mu}}{2}\mathcal{E}(t) - \sqrt{s}\lVert \nabla F(X) \rVert^2 -\gamma \sqrt{\mu} \left( \frac{\lambda \sqrt{\mu}}{2} (1+\gamma \sqrt{\mu s}) - \frac{\lambda^2}{2}\right)\lVert X-x^\ast \rVert^2 \\
             &- \gamma \sqrt{\mu}\left(\frac{\lambda}{\sqrt{\mu}}(1+\gamma \sqrt{\mu s}) - \frac{\upsilon}{2}(1+2\gamma\sqrt{\mu s})\right)(F(X)-F^\ast)
             -\gamma\frac{\sqrt{\mu }\upsilon^2}{2}\lVert \dot{X}\rVert^2 \label{continuous kinetic energy}
        \end{align}
        We have to check that some terms are negatives. We have $\lambda = \sqrt{\mu}(1+\gamma(1-\upsilon))$, and
        \begin{align}
            \frac{\lambda \sqrt{\mu}}{2} (1+\gamma \sqrt{\mu s}) - \frac{\lambda^2}{2} = \frac{\lambda}{2}\left(\sqrt{\mu}(1+\gamma\sqrt{\mu s} )- \lambda
            \right) = \frac{\lambda \mu}{2}\left( \gamma\sqrt{\mu s} -\gamma(1-\upsilon) \right) = \frac{\lambda \gamma \mu \sqrt{s}}{2}\left(\frac{1-\gamma}{2}\right) \geqslant 0
        \end{align}
        and
        \begin{align}
            \frac{\lambda}{\sqrt{\mu}}(1+\gamma \sqrt{\mu s}) - \frac{\upsilon}{2}(1+2\gamma\sqrt{\mu s}) &= \left( 1 - \frac{\gamma(1-\gamma)}{2}\sqrt{\mu s} \right)(1+\gamma \sqrt{\mu s}) - \frac{1}{2}(1 + \frac{1-\gamma}{2}\sqrt{\mu s})(1+2\gamma \sqrt{\mu s})\\
            &= 1+\gamma\sqrt{\mu s} - \frac{1}{2}(1+2\gamma \sqrt{\mu s})  - \frac{\gamma(1-\gamma)}{2}\sqrt{\mu s}(1+\gamma \sqrt{\mu s})\\
            &- \frac{1}{2} \frac{1-\gamma}{2}\sqrt{\mu s}(1+2\gamma \sqrt{\mu s})\\
            &= \frac{1}{2} - \frac{1-\gamma}{2}\sqrt{\mu s}\left(\gamma(1+\gamma\sqrt{\mu s}) + \frac{1}{2} + \gamma \sqrt{\mu s} \right)
        \end{align}
        We want to be sure that this quantity is positive. To do so, we will maximize the right term with respect to $\gamma \in [0,1]$. First, note that supposing $s \leqslant \frac{1}{L}$, we have:
        \begin{equation}
            \frac{1-\gamma}{2}\sqrt{\mu s}\left(\gamma(1+\gamma\sqrt{\mu s}) + \frac{1}{2} + \gamma \sqrt{\mu s} \right) \leqslant \frac{1-\gamma}{2}\left(\gamma(1+\gamma) + \frac{1}{2} + \gamma \right) = \frac{1}{4}(1-\gamma)(1 + 4\gamma +2\gamma^2):= g(\gamma)
        \end{equation}
        We have 
        \begin{equation}
            4g(\gamma) = (1+3\gamma - 2\gamma^2 - 2\gamma^3)
        \end{equation}
        We now want to find critical points of $g$.
        \begin{equation}
            4g'(\gamma) = 3 - 4\gamma - 6\gamma^2
        \end{equation}
        We calculate the discriminant $\Delta = 4^2+4*6*3=88$, inducing that the roots of $g'$ are
        \begin{equation}
            x_1 = -\frac{4+2\sqrt{22}}{12}, \quad x_2 = \frac{-4+2\sqrt{22}}{12}
        \end{equation}
        We clearly have $x_1<0$. $x_2$ however belongs to $[0,1]$. We evaluate numerically $g(x_2) \approx 0.44 < \frac{1}{2}$. We conclude that for all $\gamma \in [0,1]$ we have 
        \begin{equation}
            \frac{1}{2} - \frac{1-\gamma}{2}\sqrt{\mu s}\left(\gamma(1+\gamma\sqrt{\mu s}) + \frac{1}{2} + \gamma \sqrt{\mu s} \right) > 0
        \end{equation}

         All the out of parenthesis terms are negative, so we conclude that:
        \begin{equation}
            \dot{\mathcal{E}}(t) \leqslant -\gamma\frac{\sqrt{\mu}}{2}\mathcal{E}(t) \Rightarrow \mathcal{E}(t) \leqslant \mathcal{E}(0)e^{-\gamma\frac{\sqrt{\mu}}{2}t}
        \end{equation}
        Supposing $t_0 = 0$
        \paragraph{Deducing rate on $F(X(t))-F^\ast$}
        Using initial conditions, we have
        \begin{align}
            \mathcal{E}(0) &= \delta (F(X_0)-F^\ast) + \frac{1}{2}\lVert \lambda (X_0-x^\ast) - \sqrt{s}\nabla F(X_0) \rVert^2 \\
            &\leqslant \delta (F(X_0)-F^\ast) + \lambda^2\lVert X_0-x^\ast \rVert^2 + s\lVert \nabla F(X_0) \rVert^2
            \\
            &\leqslant \left( \delta + \frac{2\lambda^2(2-\gamma) }{\gamma \mu}+2Ls \right)(F(X_0)-F^\ast)
        \end{align}
      Where the third inequality uses that $F$ is $\frac{\mu \gamma}{2-\gamma}$-quadratic growth (Propostion \ref{SQC implis PL & QG}) and that $F$ is $L$-Smooth.
        Then, we have
        \begin{equation}
            F(X(t))-F^\ast \leqslant \underbrace{\left( \gamma + \frac{2\lambda^2(2-\gamma ) }{ \mu \delta}+2\gamma Ls \right)}_{:= K_0(\gamma,\mu,L,s)}\frac{1}{\gamma}(F(X_0)-F^\ast)  e^{-\gamma\frac{\sqrt{\mu}}{2}t}
        \end{equation}
        We need to check that $K_0(\gamma,\mu,L,s)$ is uniformly bounded. Note first that as $0<s \leqslant \frac{1}{L}$, we have
        \begin{equation}
            sL \leqslant 1, \quad \mu s \leqslant 1
        \end{equation}
        We bound now $\upsilon,\delta,\lambda$, that we already fixed in the proof.
        \begin{align}
           1\leq\upsilon:= 1 + \frac{1-\gamma}{2}\sqrt{\mu s} \leqslant \frac{3}{2}         \end{align}
           \begin{align}
               1\leqslant \delta := \upsilon(1 + 2 \gamma \sqrt{\mu s})\leqslant \frac{9}{2}
           \end{align}
           \begin{equation}
               \sqrt{\mu}\left(1 - \frac{1}{8} \right)\leq\lambda:= \sqrt{\mu}\left(1+\frac{\gamma(\gamma-1)}{2}\sqrt{\mu s}\right) \leqslant \sqrt{\mu}
           \end{equation}
           In the last inequality, we used the well known fact that $0\leqslant p(1-p) \leqslant \frac{1}{4}$, for all $p \in [0,1]$.\\
           We thus can explicitly compute that
           \begin{equation}
               K_0(\gamma,\mu,L,s) \leqslant \left(1 + 4 + 2  \right) = 7.
           \end{equation}
        \end{proof}

\end{document}